\documentclass[smallextended]{svjour3}       
\smartqed  
\usepackage{graphics}
\usepackage{caption}
\usepackage{mathptmx}
\usepackage{float}
\usepackage{amsmath}
\usepackage{amssymb}
\usepackage{graphicx}
\usepackage{color}
\usepackage{ifpdf}
\usepackage{url}
\usepackage{hyperref}
\usepackage{enumerate}
\usepackage{subfigure}
\graphicspath{{figures/}}
\usepackage{mathptmx, amsmath, amssymb, amsfonts, mathptmx, enumerate, color,enumitem}
\hypersetup{hidelinks} 
\usepackage{algorithm}
\usepackage{algorithmic}
\usepackage[american]{babel}
\usepackage{microtype}
\usepackage[misc]{ifsym} 

\smartqed  

\usepackage{mathptmx}      
\usepackage{latexsym}

\renewcommand{\thetable}{\arabic{section}.\arabic{table}}

\usepackage{booktabs}
\makeatletter
\renewcommand\@biblabel[1]{[#1]}
\makeatother

\begin{document}
\title{Strong convergence of inertial extragradient algorithms for solving variational inequalities and fixed point problems}
\titlerunning{Inertial extragradient algorithms}
\author{Bing Tan\and Liya Liu\and  Xiaolong Qin}
\authorrunning{B. Tan, L. Liu, X. Qin}

\institute{
Bing Tan 
\at Institute of Fundamental and Frontier Sciences, University of Electronic Science and Technology of China, Chengdu 611731, China  \\
\href{mailto:bingtan72@gmail.com}{bingtan72@gmail.com}
\and Liya Liu
\at School of Mathematical Sciences, University of Electronic Science and Technology of China, Chengdu 611731, China\\
\href{mailto:liya42@qq.com}{liya42@qq.com} 
\and
Xiaolong Qin ${(\textrm{\Letter})}$
\at Department of Mathematics, Hangzhou Normal University, Hangzhou 310000, China \\
\href{mailto:qxlxajh@163.com}{qxlxajh@163.com}
}

\date{Received: date / Accepted: date}

\maketitle

\begin{abstract}
The paper investigates two inertial extragradient algorithms for seeking a common solution to a variational inequality problem involving a monotone and Lipschitz continuous mapping and a fixed point problem with a demicontractive mapping in real Hilbert spaces. Our algorithms only need to calculate the projection on the feasible set once in each iteration. Moreover, they can work well without the prior information of the Lipschitz constant of the cost operator and do not contain any line search process. The strong convergence of the algorithms is established under suitable conditions. Some experiments are presented to illustrate the numerical efficiency of the suggested algorithms and compare them with some existing ones.

\end{abstract}
\section{Introduction}
Throughout this paper, one assumes that $H$ is a real Hilbert space with $\langle \cdot, \cdot\rangle$ and $\|\cdot\|$ as its inner product and induced norm, respectively. Let $C\subset H$ be  convex and closed, and let $P_C$ denote the metric (nearest point) projection of $H$ onto $C$. Let $A: C\rightarrow H$ be a nonlinear operator. The problem of finding the variational inequality of $x^{\dag} \in C$ (VIP for short) is considered as follows:
\begin{equation*}\label{VIP}
\langle Ax^{\dag}, x-x^{\dag}\rangle\geq 0\,, \quad  \forall x\in C\,.
\tag{\text{VIP}}
\end{equation*}
The symbol $\Omega$ represents the solution set of the problem~\eqref{VIP}.

Variational inequality problems provide a useful and indispensable tool for investigating various interesting issues emerging in many areas, such as  social, physics, engineering, economics, network analysis, medical imaging, inverse problems, transportation and many more, see, e.g., \cite{ra,QA,liu1,SYVS,AIY}. Variational inequalities theory has been proven to provide a simple, universal, and consistent structure to deal with possible problems. In the past few decades, researchers have shown tremendous interest in exploring different extensions of the variational inequality problems. Recently, various forms of computational approaches have been developed and proposed to solve the problem of variational inequalities, such as projection-based methods, hybrid steepest descent methods, and Tikhonov regularization methods.  For some related results, the reader can refer to \cite{Cho1,Cho2,dongopt,SILD,tanjnca}.

We concentrate primarily on projection-based approaches in this study. The earliest and cheapest method of projection is called the projected gradient method. This method contains only one iterative process in each iteration, and only needs to calculate one projection on the feasible set. Unfortunately, the convergence condition of this algorithm is very strong, that is, the cost operator is strongly monotone or inverse strongly monotone, which limits the wide use of the algorithm. To prevent the use of such strong assumptions, Korpelevich proposed the extragradient method (EGM)~\cite{EGM}, which can guarantee weak convergence under the condition that the cost operator is only monotone and Lipschitz continuous. Looking back on the extragradient method, it can be seen that EGM needs to evaluate the value of the cost operator twice and calculate two projections on the feasible set in each iteration. It should be remembered that when the feasible set has a complex structure, it may be very expensive to calculate the projection on the feasible set, which will further affect the efficiency of the method used. Next, let us review two notable approaches to overcome this shortcoming. The first one is the Tseng's extragradient method~\cite{tseng} (TEGM for short, it is also known as the forward-backward-forward algorithm), which is a two-step iterative method. In the second step of TEGM, an explicit formula is used to replace the second projection of EGM. So, this method calculates the projection only once on the feasible set in every iteration. Another method is the subgradient extragradient method (SEGM) proposed in~\cite{SEGM}, which is widely considered an improvement of EGM. This method replaces the second projection of EGM with the projection on a half-space. We know that projection on half-space can be calculated by an explicit formula. Therefore, SEGM greatly improves the computational efficiency of EGM.

The second problem that we are interested in is the fixed point problem (FPP). One recalls that the fixed point problem is described as follows:
\begin{equation*}\label{FPP}
\text{find } x^{\dag}\in H \text{ such that }  x^{\dag}= Tx^{\dag}\,,
\tag{\text{FPP}}
\end{equation*}
where $ T: H\rightarrow H $ is a general operator, and its fixed point set is represented as $\Gamma = \{x : Tx=x\}$. We always suppose that the fixed point set of $ T $ is non-empty, i.e., $\Gamma\neq~\emptyset$. Iterative approaches of fixed point problems of nonlinear operators have been bustling fields of study for many scholars mainly due to its applications in engineering and science recently. In recent years, iterative methods of fixed-point estimation for nonexpansive operators and demicontractive operators are studied in~\cite{TZL,liu4,aviv}.

In this paper, we are concerned about finding common solutions of variational inequality problems \eqref{VIP} and fixed point problems \eqref{FPP}. More precisely, we consider the following general problem:
\begin{equation}\label{VIPFPP}
\text{find } x^{\dag}  \text{ such that }  x^{\dag} \in  \Omega \cap \Gamma \,,
\tag{\text{VIPFPP}}
\end{equation}
where $ A: C \rightarrow H $ and $ T: H \rightarrow H $ are two nonlinear operators. The reason for exploring such problems is that they can be applied to mathematical models, and their constraints can be represented as fixed-point problems and/or variational inequality problems. In recent years, researchers have investigated and proposed many efficient iterative approaches to find common solutions for variational inequalities and fixed-point problems. We here list some of the iterative approaches to solve \eqref{VIP} and \eqref{FPP} which motivate us  to introduce our new scheme for solving \eqref{VIPFPP}. Recently, Kraikaew and Saejung~\cite{KS} proposed an algorithm called Halpern subgradient extragradient method (HSEGM) by connecting the subgradient extragradient method with the Halpern method to solve~\eqref{VIPFPP}. Their algorithm is expressed as follows:
\begin{equation*}\label{HSEGM}
\left\{\begin{aligned}
&y^{k}=P_{C}(x^{k}-\psi A x^{k})\,, \\
&H_{k}=\{x \in H:\langle x^{k}-\psi A x^{k}-y^{k}, x-y^{k}\rangle \leq 0\}\,, \\
&z^{k}=\theta_{k} x^{0}+(1-\theta_{k}) P_{H_{k}}(x^{k}-\psi A y^{k})\,, \\
&x^{k+1}=\varphi_{k} x^{k}+(1-\varphi_{k}) T z^{k}\,,
\end{aligned}\right.
\tag{\text{HSEGM}}
\end{equation*}
where $ x^{0} $ represents the initial fixed point, $\{\theta_{k}\}\subset (0, 1)$ satisfies that  $\sum_{k=1}^{\infty}\theta_{k} = \infty$, $\lim_{k\rightarrow\infty}\theta_{k} = 0$, step size $ \psi\in(0,1/L) $, mapping $A: H\rightarrow H$ is $L$-Lipschitz continuous monotone and mapping $T: H \rightarrow H$ is quasi-nonexpansive with $(I-T)$ being demiclosed at zero. Under the assumption of $\Omega \cap \Gamma  \neq \emptyset$, they proved that the sequence $\{x^{k}\}$ formulated by \eqref{HSEGM} converges to an element $u \in \Omega \cap \Gamma  $ in norm, where $u = P_{\Omega \cap \Gamma} x^{0}$. However, HSEGM converges very slowly because it uses the initial point $ x^{0} $ in each iteration. Another method to obtain strong convergence is called the viscosity method. Recently, based on the extragradient-type method and the viscosity method, Thong and Hieu~\cite{TVNA} suggested two extragradient-viscosity algorithms in a Hilbert space for solving~\eqref{VIPFPP}.  Let $ \{x_n\} $ be formulated by:
\begin{equation*}\label{VSEGM}
\left\{\begin{aligned}
&y^{k}=P_{C}(x^{k}-\psi_{k} A x^{k})\,, \\
&H_{k}=\{x \in H:\langle x^{k}-\psi_{k} A x^{k}-y^{k}, x-y^{k}\rangle \leq 0\}\,, \\
&z^{k}=P_{H_{k}}(x^{k}-\psi_{k} A y^{k})\,, \\
&x^{k+1}=\theta_{k} f(x^{k})+(1-\theta_{k})[(1-\varphi_{k}) z^{k}+\varphi_{k} T z^{k}]\,,
\end{aligned}\right.
\tag{\text{VSEGM}}
\end{equation*}
and
\begin{equation*}\label{VTEGM}
\left\{\begin{aligned}
&y^{k}=P_{C}(x^{k}-\psi_{k} A x^{k})\,, \\
&z^{k}=y^{k}-\psi_{k}(A y^{k}-A x^{k})\,, \\
&x^{k+1}=\theta_{k} f(x^{k})+(1-\theta_{k})[(1-\varphi_{k}) z^{k}+\varphi_{k} T z^{k}]\,,
\end{aligned}\right.
\tag{\text{VTEGM}}
\end{equation*}
where algorithms \eqref{VSEGM} and \eqref{VTEGM} update the step size $ \{\psi_{k}\} $ by following:
\[
\psi_{k+1}=\left\{\begin{array}{ll}
\min \left\{\frac{\phi\|x^{k}-y^{k}\|}{\|A x^{k}-A y^{k}\|}, \psi_{k}\right\}, & \text { if } A x^{k}-A y^{k} \neq 0; \\
\psi_{k}, & \text { otherwise},
\end{array}\right.
\]
where $\{\theta_{k}\}\subset (0, 1)$ satisfies that   $\sum_{k=1}^{\infty}\theta_{k} = \infty$, $\lim_{k\rightarrow\infty}\theta_{k} = 0$, $\{\varphi_{k}\}\subset (a, 1-\psi)$ for some $a> 0$ and $ \psi_0 >0$, mapping $A: H\rightarrow H$ is monotone and $L$-Lipschitz continuous, mapping $T: H \rightarrow H$ is $\vartheta$-demicontractive such that $(I-T)$ is demiclosed at zero and mapping $f: H\rightarrow H$ is $ \rho $-contraction with constant $\rho\in [0, 1)$. It was proven that, if $\Omega \cap \Gamma   \neq \emptyset$,  the sequence $\{ x_n\}$ formulated by \eqref{VSEGM}) and \eqref{VTEGM} converges strongly to $u\in \Omega \cap \Gamma  $, where $u = P_{\Omega \cap \Gamma  } \circ f(u)$. Note that \eqref{HSEGM} uses a fixed step size, that is, it needs to know the prior information of Lipschitz constant of mapping $ A $. However, \eqref{VSEGM} and \eqref{VTEGM} do not require the prior information of Lipschitz constants of the mapping, which makes them more flexible in practical applications.

It is worth noting that the above mentioned methods need to calculate at least one projection in every iteration. We know that calculating the value of the projection is equivalent to finding a solution to an optimization problem, which is computationally expensive. A natural problem appears in front of us. Is there any way to prevent calculating projections and solve variational inequalities? Indeed, Yamada~\cite{yama} proposed the hybrid steepest descent method, which is read as follows:
\[
x^{k+1} =  (I-\psi_{k} \phi S)Tx^{k}\,,
\]
where mapping $T: H\rightarrow H$ is nonexpansive, mapping $S: C\rightarrow H$ is $\kappa$-Lipschitz continuous and $\eta$-strong monotone, $0< \phi< 2\eta/\kappa^{2}$ and the sequence $\{\psi_n\}\subseteq (0, 1)$ satisfies some conditions. He proved that the formulated sequence $\{x_n\}$ converges to an element $ x^{\dag} $ in norm, which is a unique solution of the variational inequality $ \langle Ax^{\dag}, y-x^{\dag}\rangle \geq 0, \forall y \in \Gamma$.

Very recently, Tong and Tian~\cite{STEGM} combined the Tseng's extragradient method with the hybrid steepest descent method for  solving~\eqref{VIPFPP}. In addition, they use an adaptive criterion to update the step size. Indeed, the sequence $\{x_n\}$ is expressed in the following form:
\begin{equation*}\label{STEGM}
\left\{\begin{aligned}
&y^{k}=P_{C}(x^{k}-\psi_{k} A x^{k})\,, \\
&z^{k}=y^{k}-\psi_{k}(A y^{k}-A x^{k})\,, \\
&x^{k+1}=(1-\sigma \theta_{k} S) [(1-\varphi_{k}) z^{k}+\varphi_{k} T z^{k}]\,,
\end{aligned}\right.
\tag{\text{STEGM}}
\end{equation*}
where mapping $A: H\rightarrow H$ is monotone and Lipschitz continuous, mapping $T: H \rightarrow H$ is quasi-nonexpansive such that $(I-T)$ is demiclosed at zero, mapping $S: H \rightarrow H$ is $\eta$-strongly monotone and $\kappa$-Lipschitz continuous for $\eta> 0$ and $\kappa> 0$. Furthermore, for any $ \alpha>0 $, $\ell \in (0, 1)$, $\phi \in (0, 1)$, the sequence $\{\psi_{k}\}$ is selected as the maximum $\psi \in\left\{\alpha, \alpha \ell, \alpha \ell^{2}, \ldots\right\}$ satisfying $ \psi\|A x^{k}-A y^{k}\| \leq \phi\|x^{k}-y^{k}\| $.
This update criterion is called the Armijo line search rule. Under some suitable conditions, the sequence $\{x_n\}$ formulated by \eqref{STEGM} converges to $u\in \Omega \cap \Gamma$ in norm, where $u=P_{\Omega \cap \Gamma}(I-\sigma S) u$.  It should be pointed out that using the Armijo-like line search rule may require more computation time, because update the step size in each iteration requires to calculate the value of $ A $ many times.

On the other hand, problems in practical applications have the characteristics of diversity, complexity and large-scale. How to build fast and stable algorithms becomes particularly important. Recently, many scholars have developed various types of inertial algorithms by employing inertial extrapolation techniques. The inertial method is based on the discrete version of the second-order dissipative dynamical system originally proposed by Polyak~\cite{polyak}.  The main feature of the inertial type methods is that they use the previously known sequence information to generate the next iteration point. More precisely, the procedure requires two iteration steps and the second iteration step is implemented through the preceding two iterations. Note that this small change can greatly accelerate the convergence speed of the iterative algorithm. In recent years, this technique has been investigated intensively and implemented successfully to many problems, see, for instance, \cite{FISTA,iFB,SLD,fanopt}.

Encouraged and influenced by the above works, the purpose of this paper is to develop two  inertial extragradient algorithms with new step size for discovering a common solution of the variational inequality problem containing a monotone and Lipschitz continuous mapping and of the fixed point problem with a demicontractive mapping in real Hilbert spaces. Our algorithms consist of four methods: inertial method, subgradient extragradient method, Tseng's extragradient method and hybrid steepest descent method. Our two algorithms only require to calculate the projection on the feasible set once per iteration, which makes them faster. Strong convergence results of the algorithms are established without the prior information of the Lipschitz constant of the operator. Lastly, some computational tests appearing in finite and infinite dimensions are proposed to verify our theoretical results. Our algorithms develop and summarize some of the results in the literature~\cite{KS,TVNA,STEGM}.

The organizational structure of our paper is built up as follows. Some essential definitions and technical lemmas that need to be used are given in the next section. In Section~\ref{sec3}, we propose algorithms and analyze their convergence. Some numerical experiments to verify our theoretical results are presented in Section~\ref{sec4}. At last, in the final section, the paper ends with a simple summary.
\section{Preliminaries}\label{sec2}
Let $C$ be a convex and closed set in a real Hilbert space $H$. The weak convergence and strong convergence of $\{x^{k}\}_{k=1}^{\infty}$ to a point $x$ are represented by $x^{k} \rightharpoonup x$ and $x^{k} \rightarrow x$, respectively. Here we state two inequalities that need to be used in the proofs. For any $x, y \in H$ and $\theta \in \mathbb{R}$, we have
\begin{itemize}
\item $\|x+y\|^{2} \leq\|x\|^{2}+2\langle y, x+y\rangle$;
\item $\|\theta x+(1-\theta) y\|^{2}=\theta\|x\|^{2}+(1-\theta)\|y\|^{2}-\theta(1-\theta)\|x-y\|^{2}$.
\end{itemize}

For every point $x \in H$, there exists a unique nearest point in $C$, which is represented by $P_{C} x$, such that $P_{C}x:= \operatorname{argmin}\{\|x-y\|,\, y \in C\}$. $P_{C}$ is called the metric projection of $H$ onto $C$, and it is a nonexpansive mapping. The following two basic projection properties will be used many times in subsequent proofs.
\begin{itemize}
\item $ \langle x-P_{C} x, y-P_{C} x\rangle \leq 0, \, \forall y \in C $;
\item $\left\|P_{C} x-P_{C} y\right\|^{2} \leq\left\langle P_{C} x-P_{C} y, x-y\right\rangle, \,\forall y \in H$.
\end{itemize}

\begin{definition}
Assume that $T: H \rightarrow H$ is a nonlinear operator with  $\Gamma \neq \emptyset $. Then, $I-T$ is said to be demiclosed at zero if for any $\{x^{k}\}$ in $H$, the following implication holds:
\[
x^{k} \rightharpoonup x \text { and }(I-T) x^{k} \rightarrow 0 \Longrightarrow x \in \Gamma\,.
\]
\end{definition}
\begin{definition}
For any $ x, y \in H, z \in \{x: Mx=x\} $, mapping $M: H \rightarrow H$ is said to be:
\begin{itemize}
\item \emph{nonexpansive} if 
\[
\|M x-M y\| \leq \|x-y\|\,.
\]
\item \emph{$L$-Lipschitz continuous} with $L>0$ if
\[
\|M x-M y\| \leq L\|x-y\|\,.
\]
\item \emph{monotone} if
\[
\langle M x-M y, x-y\rangle \geq 0\,.
\]
\item \emph{quasi-nonexpansive} if
\[
\|M x-z\| \leq\|x-z\|\,.
\]
\item \emph{$\rho$-strictly pseudocontractive} with $0 \leq \rho <1$ if
\[
\|M x-M y\|^{2} \leq\|x-y\|^{2}+\rho\|(I-M) x-(I-M) y\|^{2}\,.
\]
\item \emph{$\vartheta$-demicontractive} with $0 \leq \vartheta <1$ if
\begin{equation}\label{eq21}
\|M x-z\|^{2} \leq\|x-z\|^{2}+\vartheta\|(I-M) x\|^{2}\,,
\end{equation}
or equivalently
\begin{equation}\label{eq23}
\langle M x-z, x-z\rangle \leq\|x-z\|^{2}+\frac{\vartheta-1}{2}\|x-M x\|^{2}\,.
\end{equation}
\end{itemize}
\end{definition}
\begin{remark}
According to the above definitions, we can easily see the following facts:
\begin{itemize}
\item The class of demicontractive mappings includes the class of quasi-nonexpansive
mappings.
\item Every strictly pseudocontractive mapping with a nonempty fixed point set is
demicontractive.
\end{itemize}
\end{remark}

The following three lemmas are crucial to prove the convergence of our algorithms.

\begin{lemma}\label{lem20}
Suppose that the mapping $S: H \rightarrow H$ is $k$-Lipschitz continuous and $\eta$-strongly monotone with $0< \eta \leq k $. Let the mapping $U: H \rightarrow H$ be nonexpansive. Given $\sigma>0 $ and $\theta\in (0,1]$, the mapping $U^{\sigma}: H \rightarrow H$ is defined by $ U^{\sigma} x=(I-\theta \sigma S)(U x),  \forall x \in H $. Then, $U^{\sigma}$ is a contraction mapping provided $\sigma<\frac{2 \eta}{k^{2}}$, i.e.,
\[
\|U^{\sigma} x-U^{\sigma} y\| \leq(1-\theta \gamma)\|x-y\|, \quad \forall x, y \in H\,,
\]
where $\gamma=1-\sqrt{1-\sigma(2 \eta-\sigma k^{2})} \in(0,1)$.
\end{lemma}
\begin{proof}
Indeed,  it follows that
\[
\begin{aligned}
\|(I-\sigma S)(U x)-(I-\sigma S)(U y)\|^{2}=&\|U x-U y\|^{2}+\sigma^{2}\|S(U x)-S(U y)\|^{2}\\
&-2 \sigma\langle U x-U y,S(U x)-S(U y) \rangle  \\
\leq &\|U x-U y\|^{2}+\sigma^{2} k^{2}\|U x-U y\|^{2}-2 \sigma \eta\|U x-U y\|^{2} \\
=&(1-\sigma(2 \eta-\sigma k^{2}))\|U x-U y\|^{2}\,.
\end{aligned}
\]
It follows from $0<\eta \leq k$ that
\[
1-\sigma(2 \eta-\sigma k^{2})=(\sigma k-\frac{\eta}{k})^{2}+1-\frac{\eta^{2}}{k^{2}} \geq 0\,,
\]
Therefore, we get
\begin{equation*}
\|(I-\sigma S)(U x)-(I-\sigma S)(U y)\| \leq \sqrt{1-\sigma(2 \eta-\sigma k^{2})}\|x-y\|\,.
\end{equation*}
From the definition of $ U^{\sigma} x $, one has
\[
\begin{aligned}
\|U^{\sigma} x-U^{\sigma} y\| &=\|(I-\theta \sigma S)(U x)-(I-\theta \sigma S)(U y)\| \\
&=\|\theta[(I-\sigma S)(U x)-(I-\sigma S)(U y)]+(1-\theta)(U x-U y)\| \\
&\leq \theta\|(I-\sigma S)(U x)-(I-\sigma S)(U y)\|+(1-\theta) \| x-y) \|\,.
\end{aligned}
\]
Thus, we conclude that
\[
\|U^{\sigma} x-U^{\sigma} y\| \leq(1-\theta \gamma)\|x-y\|\,.
\]
where $\gamma=1-\sqrt{1-\sigma(2 \eta-\sigma k^{2})} \in(0,1)$ with $0<\eta \leq k$ and $\sigma<\frac{2 \eta}{k^{2}}$.
\end{proof}
\begin{lemma}[\cite{KS}]\label{lem21}
Assume that mapping $A : H \rightarrow H$ is monotone and $ L $-Lipschitz continuous  on $ C $. Set $T=P_{C}(I-\phi A)$, where $\phi>0$. If $\{x^{k}\} \subset H$ satisfying $x^{k} \rightharpoonup u$ and $x^{k}-T x^{k} \rightarrow 0$. Then $u \in \Omega=\Gamma$.
\end{lemma}
\begin{lemma}[\cite{SY2012}]\label{lem23}
Let $\{a^{k}\}$ be a nonnegative real number sequence. The sequence $\{\theta_{k}\}\subset (0,1)$ satisfies $\sum_{k=1}^{\infty} \theta_{k}=\infty$. Assume that the following inequality holds:
\[
a^{k+1} \leq(1-\theta_{k}) a^{k}+\theta_{k} b^{k}, \quad \forall k \geq 1,
\]
where $\{b^{k}\}$ is  a real number sequence  such that  $\limsup _{i \rightarrow \infty} b^{k_{i}} \leq 0$ for every subsequence $\{a^{k_{i}}\}$ of $\{a^{k}\}$ satisfying $\lim \inf _{i \rightarrow \infty}$
$(a^{k_{i}+1}-a^{k_{i}}) \geq~0$. Then $\lim _{k \rightarrow \infty} a^{k}=0$.
\end{lemma}

\section{Strong convergence of two inertial algorithms}\label{sec3}
In this section, we present two inertial extragradient methods with new step size for searching a common solution of variational inequality problems and fixed point problems and analyze their convergence. The advantage of our two iterative algorithms is that we only need to calculate the projection on the feasible set once in each iteration, and we do not require to know the prior information of the Lipschitz constant of the mapping. Before starting to introduce the algorithms, we first assume that our iteration scheme meets the following five conditions.
\begin{enumerate}[label=(C\arabic*)]
\item The mapping $A: H \rightarrow H$ is monotone and $L$-Lipschitz continuous on $H$. \label{con1}
\item The mapping $T: H \rightarrow H$ is $\vartheta$-demicontractive such that $(I-T)$ is demiclosed at zero. \label{con2}
\item The solution set of our problem is non-empty, i.e., $ \Omega \cap \Gamma \neq \emptyset $. \label{con3}
\item The mapping $S: H \rightarrow H$ is $\eta$-strongly monotone and $ k $-Lipschitz continuous, where $ \eta $ and $ k $ are positive numbers.
\item Let $ \{\zeta_{k}\} $ be a positive sequence satisfies $\lim_{k \rightarrow \infty} \frac{\zeta_{k}}{\theta_{k}}=0$, where $ \{\theta_{k}\}\subset (0,1) $ such that $\sum_{k=1}^{\infty} \theta_{k}=\infty$ and $\lim _{k \rightarrow \infty} \theta_{k}=0$. Let $\left\{\varphi_{k}\right\}$ be a real sequence such that $\varphi_{k} \subset(a,1-\vartheta)$ for some $a>0$. \label{con5}
\end{enumerate}

\subsection{The self adaptive inertial subgradient extragradient algorithm}
So far, we can state our first self-adaptive iterative algorithm, which is motivated by the inertial subgradient extragradient method and the hybrid steepest descent method with a new step size. Our algorithm is described as follows:
\begin{algorithm}[h]
\caption{The self adaptive inertial subgradient extragradient algorithm}
\label{alg1}
\begin{algorithmic}
\STATE {\textbf{Initialization:} Give  $ \xi>0 $, $\psi_{1}>0$, $\phi \in(0,1)$, $\sigma \in(0, \frac{2 \eta}{k^{2}})$. Let $x^{0},x^{1} \in H$ be two initial points.}
\STATE \textbf{Iterative Steps}: Calculate the next iteration point $ x^{k+1} $ as follows:
\STATE \textbf{Step 1.} Given two previously known iteration points $x^{k-1}$ and $x^{k}(k \geq 1) $. Calculate
\[u^{k}=x^{k}+\xi_{k}(x^{k}-x^{k-1})\,,\]
where
\begin{equation}\label{alpha}
\xi_{k}=\left\{\begin{array}{ll}
\min \bigg\{\dfrac{\zeta_{k}}{\|x^{k}-x^{k-1}\|}, \xi\bigg\}, & \text { if } x^{k} \neq x^{k-1}; \\
\xi, & \text { otherwise}.
\end{array}\right.
\end{equation}
\STATE \textbf{Step 2.} Calculate
\[y^{k}=P_{C}(u^{k}-\psi_{k} A u^{k})\,,\]
\STATE \textbf{Step 3.} Calculate
\[z^{k}=P_{H_{k}}(u^{k}-\psi_{k} A y^{k})\,,\]
where $ H_{k}:=\{x \in H \mid\langle u^{k}-\psi_{k} A u^{k}-y^{k}, x-y^{k}\rangle \leq 0\} $.
\STATE \textbf{Step 4.} Calculate
\[x^{k+1}=(1-\sigma \theta_{k} S) q^{k}\,,\]
where $q^{k}=(1-\varphi_{k}) z^{k}+\varphi_{k} T z^{k}$, and update
\begin{equation}\label{lambda2}
\psi_{k+1}=\left\{\begin{array}{ll}
\min \left\{\dfrac{\phi\|u^{k}-y^{k}\|}{\|A u^{k}-A y^{k}\|}, \psi_{k}\right\}, & \text { if } A u^{k}-A y^{k} \neq 0; \\
\psi_{k}, & \text { otherwise}.
\end{array}\right.
\end{equation}
\end{algorithmic}
\end{algorithm}
\begin{remark}\label{rem31}
It follows from \eqref{alpha} and Condition~\ref{con5} that 
\[
\lim _{k \rightarrow \infty} \frac{\xi_{k}}{\theta_{k}}\|x^{k}-x^{k-1}\|=0\,.
\]
Indeed, we obtain $\xi_{k}\|x^{k}-x^{k-1}\| \leq \zeta_{k}$ for all $ k $, which together with $\lim _{k \rightarrow \infty} \frac{\zeta_{k}}{\theta_{k}}=0$ implies that
\[
\lim _{k \rightarrow \infty} \frac{\xi_{k}}{\theta_{k}}\|x^{k}-x^{k-1}\| \leq \lim _{k \rightarrow \infty} \frac{\zeta_{k}}{\theta_{k}}=0\,.
\]
\end{remark}
Before we begin to state our main theorems, the following two lemmas are very helpful for the convergence analysis of the algorithms.
\begin{lemma}\label{lem31}
The sequence $\left\{\psi_{k}\right\}$ formulated by \eqref{lambda2} is  nonincreasing and satisfies
\[
\lim _{k \rightarrow \infty} \psi_{k}=\psi \geq \min \Big\{\psi_{1}, \frac{\phi}{L}\Big\}\,.
\]
\end{lemma}
\begin{proof}
It follows from \eqref{lambda2} that $\psi_{k+1} \leq \psi_{k}$ for all $k \in \mathbb{N} $. Hence, $\left\{\psi_{k}\right\}$ is nonincreasing. Furthermore, we get $\|A u^{k}-A y^{k}\| \leq L\|u^{k}-y^{k}\|$ since $A$ is $L$-Lipschitz continuous. Consequently, we can show that
\[
\phi \frac{\|u^{k}-y^{k}\|}{\|A u^{k}-A y^{k}\|} \geq \frac{\phi}{L}, \,\,\text {  if  }\,\, A u^{k} \neq A y^{k}\,.
\]
In view of \eqref{lambda2},  it follows that
\[
\psi_{k} \geq \min \Big\{\psi_{1}, \frac{\phi}{L}\Big\}\,.
\]
Thus, from the sequence $ \{\psi_{k}\} $  is nonincreasing and lower bounded, we get that $\lim _{k \rightarrow \infty} \psi_{k}=\psi \geq \min \big\{\psi_{1}, \frac{\phi}{L}\big\}$.
\end{proof}
\begin{lemma}[\cite{tanarxiv}]\label{lem32}
Suppose that Conditions \ref{con1} and \ref{con3} hold. Let sequence $\{z^{k}\}$ be formulated by Algorithm~\ref{alg1}. Then, for any $x^{\dag} \in \Omega$, we get
\begin{equation}\label{q}
\|z^{k}-x^{\dag}\|^{2} \leq\|u^{k}-x^{\dag}\|^{2}-\big(1-\phi \frac{\psi_{k}}{\psi_{k+1}}\big)\|y^{k}-u^{k}\|^{2}-\big(1-\phi \frac{\psi_{k}}{\psi_{k+1}}\big)\|z^{k}-y^{k}\|^{2}\,.
\end{equation}
\end{lemma}

\begin{theorem}\label{thm31}
Suppose that Conditions \ref{con1}--\ref{con5} hold. Then the iterative sequence $\{x^{k}\}$ formulated by Algorithm~\ref{alg1} converges to an element $ x^{\dag}\in \Omega \cap \Gamma $ in norm, where $ x^{\dag} = P_{\Omega \cap \Gamma}(I-\sigma S) x^{\dag} $.
\end{theorem}
\begin{proof}
According to Lemma~\ref{lem20}, we get that $(I-\sigma S)$ is a contractive mapping. Therefore, $P_{\Omega \cap \Gamma  }(I-\sigma S)$ is also a contraction mapping. By means of the Banach contraction principle, one concludes that there exists a unique point $x^{\dag} \in H$ such that $x^{\dag}=P_{\Omega \cap \Gamma  }(I-\sigma S) x^{\dag} $. Let $x^{\dag} \in\Omega \cap \Gamma$.

\noindent{\bf Claim 1.} The sequence $\{x^{k}\}$ is bounded. On account of Lemma~\ref{lem31}, we see that $\lim _{k \rightarrow \infty}(1-\phi \frac{\psi_{k}}{\psi_{k+1}})=1-\phi>0$. Hence, $\exists \, k_{0} \in \mathbb{N}$ such that
\begin{equation}\label{eqf}
1-\phi \frac{\psi_{k}}{\psi_{k+1}}>0, \quad  \forall k \geq k_{0}\,.
\end{equation}
Combining Lemma~\ref{lem32} and \eqref{eqf},  it follows that
\begin{equation}\label{eqg}
\|z^{k}-x^{\dag}\| \leq\|u^{k}-x^{\dag}\|, \quad \forall k \geq k_{0}\,.
\end{equation}
According to the definition of $ u^{k} $, we can write
\begin{equation}\label{eqi}
\begin{aligned}
\|u^{k}-x^{\dag}\| 
& \leq\|x^{k}-x^{\dag}\|+\xi_{k}\|x^{k}-x^{k-1}\| \\
&=\|x^{k}-x^{\dag}\|+\theta_{k} \cdot \frac{\xi_{k}}{\theta_{k}}\|x^{k}-x^{k-1}\|\,.
\end{aligned}
\end{equation}
From Remark~\ref{rem31}, one sees that $\frac{\xi_{k}}{\theta_{k}}\|x^{k}-x^{k-1}\| \rightarrow 0$. Therefore, there exists a constant $Q_{1}>0$ such that
\begin{equation}\label{eqj}
\frac{\xi_{k}}{\theta_{k}}\|x^{k}-x^{k-1}\| \leq Q_{1}, \quad \forall k \geq 1\,.
\end{equation}
By \eqref{eqg}, \eqref{eqi} and \eqref{eqj}, we have
\begin{equation}\label{eqm}
\|z^{k}-x^{\dag}\| \leq\|u^{k}-x^{\dag}\| \leq\|x^{k}-x^{\dag}\|+\theta_{k} Q_{1},\quad \forall k \geq k_{0}\,.
\end{equation}
On the other hand, from the definition of $ q^{k} $, \eqref{eq21} and \eqref{eq23}, we get
\begin{equation}\label{eqqa}
\begin{aligned}
\|q^{k}-x^{\dag}\|^{2} 
=&\|(1-\varphi_{k})(z^{k}-x^{\dag})+\varphi_{k}(T z^{k}-x^{\dag})\|\\
=&(1-\varphi_{k})^{2}\|z^{k}-x^{\dag}\|^{2}+\varphi_{k}^{2}\|T z^{k}-x^{\dag}\|^{2}+2(1-\varphi_{k}) \varphi_{k}\langle T z^{k}-x^{\dag}, z^{k}-x^{\dag}\rangle\\ \leq&(1-\varphi_{k})^{2}\|z^{k}-x^{\dag}\|^{2}+\varphi_{k}^{2}\|z^{k}-x^{\dag}\|^{2}+\varphi_{k}^{2} \vartheta\|T z^{k}-z^{k}\|^{2} \\
&+2(1-\varphi_{k}) \varphi_{k}\big[\|z^{k}-x^{\dag}\|^{2}-\frac{1-\vartheta}{2}\|T z^{k}-z^{k}\|^{2}\big] \\
=&\|z^{k}-x^{\dag}\|^{2}+\varphi_{k}[\varphi_{k}-(1-\vartheta)]\|T z^{k}-z^{k}\|^{2}\,.
\end{aligned}
\end{equation}
In view of $ \{\varphi_{k}\}\subset(0,1-\vartheta) $ and  \eqref{eqm}, we get
\begin{equation}\label{eqw}
\|q^{k}-x^{\dag}\| \leq\|u^{k}-x^{\dag}\|\leq\|x^{k}-x^{\dag}\|+\theta_{k} Q_{1},\quad \forall k \geq k_{0}\,.
\end{equation}
Therefore,  on account of Lemma~\ref{lem20} and \eqref{eqw}, we have
\[
\begin{aligned}
\|x^{k+1}-x^{\dag}\| 
& = \|(I-\sigma \theta_{k} S) q^{k}-(I-\sigma \theta_{k} S) p - \sigma \theta_{k} S x^{\dag}\|\\
& \leq\|(I-\sigma \theta_{k} S) q^{k}-(I-\sigma \theta_{k} S) p\|+\sigma \theta_{k}\|S x^{\dag}\| \\
& \leq(1-\gamma \theta_{k})\|q^{k}-x^{\dag}\|+\sigma \theta_{k}\|S x^{\dag}\| \\
& \leq(1-\gamma \theta_{k})\|x^{k}-x^{\dag}\|+\gamma \theta_{k} \frac{\sigma}{\gamma}\|S x^{\dag}\| +\gamma \theta_{k} \frac{Q_{1}}{\gamma}  \\
& \leq \max \Big\{\|x^{k}-x^{\dag}\|, \frac{\sigma\|S x^{\dag}\|+Q_{1}}{\gamma}\Big\} \\
& \leq \cdots  \leq \max \Big\{\|x^{0}-x^{\dag}\|, \frac{\sigma\|S x^{\dag}\|+Q_{1}}{\gamma}\Big\}\,,
\end{aligned}
\]
where $\gamma=1-\sqrt{1-\sigma(2 \eta-\sigma k^{2})} \in(0,1)$. This means that the sequence $\{x^{k}\}$ is bounded. Thus, the sequences $\{y^{k}\},\{z^{k}\},\{q^{k}\}$ and $\left\{(I-\sigma S) x^{k}\right\}$ are also bounded.

\noindent{\bf Claim 2.}
\[
\begin{aligned}
&\quad\varphi_{k}[1-\vartheta-\varphi_{k}]\|z^{k}-T z^{k}\|^{2}+\big(1-\phi \frac{\psi_{k}}{\psi_{k+1}}\big)\|y^{k}-u^{k}\|^{2}+\big(1-\phi \frac{\psi_{k}}{\psi_{k+1}}\big)\|z^{k}-y^{k}\|^{2}\\ &\leq\|x^{k}-x^{\dag}\|^{2}-\|x^{k+1}-x^{\dag}\|^{2}+\theta_{k} Q_{4}\,, \quad \forall k \geq k_{0}
\end{aligned}
\]
for some $ Q_{4}>0 $. Indeed,  on account of Lemma~\ref{lem20} and \eqref{eqqa}, it follows that
\begin{equation}\label{eqaa}
\begin{aligned}
\|x^{k+1}-x^{\dag}\|^2 
& = \|(I-\sigma \theta_{k} S) q^{k}-(I-\sigma \theta_{k} S) p - \sigma \theta_{k} S x^{\dag}\|^2\\ 
& \leq\|(I-\sigma \theta_{k} S) q^{k}-(I-\sigma \theta_{k} S) p\|^2- 2\sigma\theta_{k}\langle S x^{\dag}, x^{k+1}-x^{\dag}\rangle \\
& \leq(1-\gamma \theta_{k})^2\|q^{k}-x^{\dag}\|^2+2\sigma\theta_{k} \langle S x^{\dag}, x^{\dag}-x^{k+1}\rangle\\
& \leq \|q^{k}-x^{\dag}\|^2+ \theta_{k}Q_{2}\\
&\leq \|z^{k}-x^{\dag}\|^{2}+\varphi_{k}[\varphi_{k}-(1-\vartheta)]\|T z^{k}-z^{k}\|^{2}+ \theta_{k}Q_{2}
\end{aligned}
\end{equation}
for some $Q_{2}>0$. In the light of Lemma~\ref{lem32}, one has
\begin{equation}\label{eqn}
\begin{aligned}
\|x^{k+1}-x^{\dag}\|^{2} \leq&\|u^{k}-x^{\dag}\|^{2}-\big(1-\phi \frac{\psi_{k}}{\psi_{k+1}}\big)\|y^{k}-u^{k}\|^{2}-\big(1-\phi \frac{\psi_{k}}{\psi_{k+1}}\big)\|z^{k}-y^{k}\|^{2}\\
&+\varphi_{k}[\varphi_{k}-(1-\vartheta)]\|T z^{k}-z^{k}\|^{2}+ \theta_{k}Q_{2}\,.
\end{aligned}
\end{equation}
In view of \eqref{eqm}, we have
\begin{equation}\label{eqo}
\begin{aligned}
\|u^{k}-x^{\dag}\|^{2} & \leq(\|x^{k}-x^{\dag}\|+\theta_{k} Q_{1})^{2} \\
&=\|x^{k}-x^{\dag}\|^{2}+\theta_{k}(2 Q_{1}\|x^{k}-x^{\dag}\|+\theta_{k} Q_{1}^{2}) \\
& \leq\|x^{k}-x^{\dag}\|^{2}+\theta_{k} Q_{3}
\end{aligned}
\end{equation}
for some $Q_{3}>0 $. From \eqref{eqn} and \eqref{eqo}, we get
\begin{equation*}
\begin{aligned}
\|x^{k+1}-x^{\dag}\|^{2} \leq&\|x^{k}-x^{\dag}\|^{2}-\big(1-\phi \frac{\psi_{k}}{\psi_{k+1}}\big)\|y^{k}-u^{k}\|^{2}-\big(1-\phi \frac{\psi_{k}}{\psi_{k+1}}\big)\|z^{k}-y^{k}\|^{2}\\
&+\varphi_{k}[\varphi_{k}-(1-\vartheta)]\|T z^{k}-z^{k}\|^{2}+ \theta_{k}Q_{2} + \theta_{k}Q_{3}\,.
\end{aligned}
\end{equation*}
which yields
\[
\begin{aligned}
&\quad\varphi_{k}[1-\vartheta-\varphi_{k}]\|z^{k}-T z^{k}\|^{2}+\big(1-\phi \frac{\psi_{k}}{\psi_{k+1}}\big)\|y^{k}-u^{k}\|^{2}+\big(1-\phi \frac{\psi_{k}}{\psi_{k+1}}\big)\|z^{k}-y^{k}\|^{2}\\ &\leq\|x^{k}-x^{\dag}\|^{2}-\|x^{k+1}-x^{\dag}\|^{2}+\theta_{k} Q_{4}\,, \quad \forall k \geq k_{0}\,,
\end{aligned}
\]
where $Q_{4}:=Q_{2}+Q_{3}$.

\noindent{\bf Claim 3.}
\begin{equation*}
\|x^{k+1}-x^{\dag}\|^2
\leq(1-\gamma \theta_{k})\|x^{k}-x^{\dag}\|^2+\gamma\theta_{k} \big[\frac{2\sigma}{\gamma}\langle S x^{\dag}, x^{\dag}-x^{k+1}\rangle+\frac{3Q\xi_{k}}{\gamma\theta_{k} }\|x^{k}-x^{k-1}\|\big],\, \forall k \geq k_{0}\,,
\end{equation*}
Indeed, by the definition of $ \{u^{k}\} $, one obtains
\begin{equation}\label{eqk}
\begin{aligned}
\|u^{k}-x^{\dag}\|^{2}  
& = \|x^{k}+\xi_{k}(x^{k}-x^{k-1})-x^{\dag} \| \\
&=\|x^{k}-x^{\dag}\|^{2}+2 \xi_{k}\langle x^{k}-x^{\dag}, x^{k}-x^{k-1}\rangle+\xi_{k}^{2}\|x^{k}-x^{k-1}\|^{2} \\
&\leq\|x^{k}-x^{\dag}\|^{2} + 3Q\xi_{k}\|x^{k}-x^{k-1}\|\,,
\end{aligned}
\end{equation}
where $Q:=\sup _{k \in \mathbb{N}}\left\{\|x^{k}-x^{\dag}\|, \xi\|x^{k}-x^{k-1}\|\right\}>0$.
Using \eqref{eqw} and \eqref{eqaa}, we obtain
\begin{equation}\label{eqaz}
\|x^{k+1}-x^{\dag}\|^2
\leq(1-\gamma \theta_{k})\|u^{k}-x^{\dag}\|^2+2\sigma\theta_{k} \langle S x^{\dag}, x^{\dag}-x^{k+1}\rangle\,.
\end{equation}
Substituting \eqref{eqk} into \eqref{eqaz}, it follows that
\begin{equation*}
\|x^{k+1}-x^{\dag}\|^2
\leq(1-\gamma \theta_{k})\|x^{k}-x^{\dag}\|^2+\gamma\theta_{k} \big[\frac{2\sigma}{\gamma}\langle S x^{\dag}, x^{\dag}-x^{k+1}\rangle+\frac{3Q\xi_{k}}{\gamma\theta_{k} }\|x^{k}-x^{k-1}\|\big],\, \forall k \geq k_{0}\,,
\end{equation*}

\noindent{\bf Claim 4.} The sequence $\big\{\|x^{k}-x^{\dag}\|^{2}\big\}$ converges to zero. From Lemma~\ref{lem23}, we need to show that $\lim \sup _{i \rightarrow \infty}\langle S x^{\dag}, x^{\dag}-x^{k_{i}+1}\rangle \leq 0$ for every subsequence $\{\|x^{k_{i}}-x^{\dag}\|\}$ of $\{\|x^{k}-x^{\dag}\|\}$ satisfying
\[
\liminf _{i \rightarrow \infty}(\|x^{k_{i}+1}-x^{\dag}\|-\|x^{k_{i}}-x^{\dag}\|) \geq 0\,.
\]

For this purpose, one assumes that $\{\|x^{k_{i}}-x^{\dag}\|\}$ is a subsequence of $\{\|x^{k}-x^{\dag}\|\}$ such that $\liminf _{i \rightarrow \infty}(\|x^{k_{i}+1}-x^{\dag}\|-\|x^{k_{i}}-x^{\dag}\|) \geq 0 $. We obtain
\[\begin{aligned}
&\quad\lim _{i \rightarrow \infty} \inf (\|x^{k_{i}+1}-x^{\dag}\|^{2}-\|x^{k_{i}}-x^{\dag}\|^{2}) \\
&=\liminf _{i \rightarrow \infty}[(\|x^{k_{i}+1}-x^{\dag}\|-\|x^{k_{i}}-x^{\dag}\|)(\|x^{k_{i}+1}-x^{\dag}\|+\|x^{k_{i}}-x^{\dag}\|)] \geq 0\,.
\end{aligned}
\]
From Claim 2 and Condition~\ref{con5}, it follows that
\[
\begin{aligned}
&\quad\limsup_{i \rightarrow \infty}\big[\big(1-\phi \frac{\psi_{k_{i}}}{\psi_{k_{i}+1}}\big)\|y^{k_{i}}-u^{k_{i}}\|^{2}+\big(1-\phi \frac{\psi_{k_{i}}}{\psi_{k_{i}+1}}\big)\|z^{k_{i}}-y^{k_{i}}\|^{2}\big. \\
&\quad+\big. \varphi_{k_{i}}(1-\vartheta-\varphi_{k_{i}})\|T z^{k_{i}}-z^{k_{i}}\|^{2}\big] \\
& \leq \limsup _{i \rightarrow \infty} [\|x^{k_{i}}-x^{\dag}\|^{2}-\|x^{k_{i}+1}-x^{\dag}\|^{2}+\theta_{k_{i}} Q_{4}] \\
& \leq \limsup _{i \rightarrow \infty} [\|x^{k_{i}}-x^{\dag}\|^{2}-\|x^{k_{i}+1}-x^{\dag}\|^{2}] +\limsup _{i \rightarrow \infty}\theta_{k_{i}} Q_{4} \\
&=-\liminf _{i \rightarrow \infty}[\|x^{k_{i}+1}-x^{\dag}\|^{2}-\|x^{k_{i}}-x^{\dag}\|^{2}] \\
& \leq 0\,.
\end{aligned}
\]
Thus, we obtain the following results:
\begin{equation}\label{v}
\lim _{i \rightarrow \infty}\|y^{k_{i}}-u^{k_{i}}\|=0, \,\,\lim _{i \rightarrow \infty}\|z^{k_{i}}-y^{k_{i}}\|=0 \text{ and }  \lim _{i \rightarrow \infty}\|T z^{k_{i}}-z^{k_{i}}\|=0\,.
\end{equation}
Therefore, we have
\begin{equation}\label{eqp}
\lim _{i \rightarrow \infty}\|z^{k_{i}}-u^{k_{i}}\|\leq \lim _{i \rightarrow \infty}\|z^{k_{i}}-y^{k_{i}}\|+ \lim _{i \rightarrow \infty}\|y^{k_{i}}-u^{k_{i}}\| =0\,,
\end{equation}
and
\begin{equation}\label{eqr}
\lim _{i \rightarrow \infty}\|x^{k_{i}}-u^{k_{i}}\|=\lim _{i \rightarrow \infty}\xi_{k_{i}}\|x^{k_{i}}-x^{k_{i}-1}\|=\lim _{i \rightarrow \infty}\theta_{k_{i}} \cdot \frac{\xi_{k_{i}}}{\theta_{k_{i}}}\|x^{k_{i}}-x^{k_{i}-1}\| = 0 \,.
\end{equation}
Combining \eqref{eqp} and \eqref{eqr}, we obtain
\begin{equation}\label{eqqz}
\lim _{i \rightarrow \infty}\|z^{k_{i}}-x^{k_{i}}\|\leq \lim _{i \rightarrow \infty}\|z^{k_{i}}-u^{k_{i}}\|+ \lim _{i \rightarrow \infty}\|u^{k_{i}}-x^{k_{i}}\| =0\,.
\end{equation}
From $q^{k_{i}}=(1-\varphi_{k_{i}}) z^{k_{i}}+\varphi_{k} T z^{k_{i}}$, one sees that
\begin{equation*}
\|q^{k_{i}}-z^{k_{i}}\| \leq \varphi_{k_{i}}\|T z^{k_{i}}-z^{k_{i}}\| \leq(1-\vartheta)\|T z^{k_{i}}-z^{k_{i}}\|\,.
\end{equation*}
In view of \eqref{v}, we get
\begin{equation}\label{h}
\lim _{i \rightarrow \infty}\|q^{k_{i}}-z^{k_{i}}\|=0\,.
\end{equation}
Moreover,  
\begin{equation}\label{eqq}
\|x^{k_{i}+1}-q^{k_{i}}\|=\sigma\theta_{k_{i}}\|S q^{k_{i}}\| \rightarrow 0 \text { as } k \rightarrow \infty\,.
\end{equation}
By \eqref{eqqz}, \eqref{h} and \eqref{eqq}, we obtain
\begin{equation}\label{equ}
\|x^{k_{i}+1}-x^{k_{i}}\| \leq\|x^{k_{i}+1}-q^{k_{i}}\|+\|q^{k_{i}}-z^{k_{i}}\|+\|z^{k_{i}}-x^{k_{i}}\| \rightarrow 0 \text { as } k \rightarrow \infty\,.
\end{equation}
It follows from $\{x^{k_{i}}\}$ is bounded that there is a subsequence $\{x^{k_{i_{j}}}\}$ of $\{x^{k_{i}}\}$ such that $x^{k_{i_{j}}} \rightharpoonup z$, where $z \in H$. From~\eqref{eqr}, we get $u^{k_{i}} \rightharpoonup z$ as $k \rightarrow \infty$. This together with $\lim _{i \rightarrow \infty}\|u^{k_{i}}-y^{k_{i}}\|=0$ and Lemma~\ref{lem21} implies that $z \in \Omega \cap \Gamma$. According to the definition of  $x^{\dag}=P_{\Omega \cap \Gamma}(I-\sigma S) x^{\dag}$, using the property of projection, one has $ \langle(I-\sigma S) x^{\dag}-x^{\dag}, z-x^{\dag}\rangle \leq 0 $. Thus,  we get
\begin{equation}\label{eqt}
\limsup _{i \rightarrow \infty}\langle S x^{\dag}, x^{\dag}-x^{k_{i}}\rangle=\lim _{j \rightarrow \infty}\langle S x^{\dag}, x^{\dag}-x^{k_{i_{j}}}\rangle=\langle S x^{\dag}, x^{\dag}-z\rangle\leq 0\,.
\end{equation}
Combining \eqref{equ} and \eqref{eqt}, we obtain
\begin{equation}\label{eqv}
\limsup _{i \rightarrow \infty}\langle S x^{\dag}, x^{\dag}-x^{k_{i}+1}\rangle =\limsup _{i \rightarrow \infty} \langle S x^{\dag}, x^{\dag}-x^{k_{i}}\rangle\leq 0\,.
\end{equation}
Hence, combining \eqref{eqv}, $\lim _{k \rightarrow \infty} \frac{\xi_{k}}{\theta_{k}}\|x^{k}-x^{k-1}\|=0$, Claim 3 and Lemma~\ref{lem23}, it follows that $\lim _{k \rightarrow \infty}\|x^{k}-x^{\dag}\|=~0$, namely, $ x^{k} \rightarrow x^{\dag} $. We have thus proved the theorem.
\end{proof}

Next, we state a particular situation of Algorithm~\ref{alg1}. When $ S(x)=x-x^{0} $ ($ x^{0} $ is an initial point) in Theorem~\ref{thm31}. It can be easily  checked that mapping $ S: H \rightarrow H $ is strongly monotone and Lipschitz continuous with modulus $ \eta = k =1 $. In this situation, by selecting $ \sigma=1 $, we obtain a new self-adaptive inertial Mann-type Halpern subgradient extragradient algorithm to solve \eqref{VIPFPP}. More specifically, we have the following result.
\begin{corollary}\label{coro1}
Suppose that mapping $A: H \rightarrow H$ is $L$-Lipschitz continuous monotone and mapping $T: H \rightarrow H$ is $\vartheta$-demicontractive such that $(I-T)$ is demiclosed at zero. Give $ \xi>0 $, $\psi_{1}>0$, $\phi \in(0,1)$. Let sequence $ \{\zeta_{k}\} $ be positive numbers such that $\lim_{k \rightarrow \infty} \frac{\zeta_{k}}{\theta_{k}}=0$, where $ \{\theta_{k}\}\subset (0,1) $ satisfies $\lim _{k \rightarrow \infty} \theta_{k}=0$ and  $\sum_{k=0}^{\infty} \theta_{k}=\infty $. Let $\left\{\varphi_{k}\right\}$ be a real sequence such that $\varphi_{k} \subset(a,1-\vartheta)$ for some $a>0$. With two start points $x^{0},x^{1} \in H$, the sequence $\{x^{k}\}$ is defined by
\begin{equation}\label{eqx}
\left\{\begin{aligned}
&u^{k}=x^{k}+\xi_{k}(x^{k}-x^{k-1})\,,\\
&y^{k}=P_{C}(u^{k}-\psi_{k} A u^{k})\,, \\
&z^{k}=P_{H_{k}}(u^{k}-\psi_{k} A y^{k})\,, \\
&H_{k}:=\{x \in H \mid\langle u^{k}-\psi_{k} A u^{k}-y^{k}, x-y^{k}\rangle \leq 0\}\,, \\
&x^{k+1}=\theta_{k} x^{0}+(1-\theta_{k}) [(1-\varphi_{k}) z^{k}+\varphi_{k} T z^{k}]\,, 
\end{aligned}\right.
\end{equation}
where inertial parameter $ \xi_{k} $ and step size $ \psi_{k} $ are defined \eqref{alpha} and \eqref{lambda2}, respectively. Then the iteration sequence $\{x^{k}\}$ formulated by \eqref{eqx} converges to $x^{\dag} \in \Omega \cap \Gamma$ in norm, where $x^{\dag}=P_{\Omega\cap \Gamma} x^{0}$.
\end{corollary}

\subsection{The self adaptive inertial Tseng's extragradient algorithm}
Next, we introduce a new self-adaptive inertial Tseng's extragradient algorithm to solve~\eqref{VIPFPP}. The advantage of this algorithm is that only one projection needs to be calculated in each iteration, and it can work without prior information of Lipschitz constant of the mapping. This algorithm is read as follows.
\begin{algorithm}[H]
\caption{The self adaptive inertial Tseng's extragradient algorithm}
\label{alg2}
\begin{algorithmic}
\STATE {\textbf{Initialization:} Give  $ \xi>0 $, $\psi_{1}>0$, $\phi \in(0,1)$, $\sigma \in(0, \frac{2 \eta}{k^{2}})$. Let $x^{0},x^{1} \in H$ be two initial points.}
\STATE \textbf{Iterative Steps}: Calculate the next iteration point $ x^{k+1} $ as follows:
\STATE \textbf{Step 1.} Given two previously known iteration points $x^{k-1}$ and $x^{k}(k \geq 1) $. Calculate
\[u^{k}=x^{k}+\xi_{k}(x^{k}-x^{k-1})\,,\]
where
\begin{equation*}\label{alpha1}
\xi_{k}=\left\{\begin{array}{ll}
\min \bigg\{\dfrac{\zeta_{k}}{\|x^{k}-x^{k-1}\|}, \xi\bigg\}, & \text { if } x^{k} \neq x^{k-1}; \\
\xi, & \text { otherwise}.
\end{array}\right.
\end{equation*}
\STATE \textbf{Step 2.} Calculate
\[y^{k}=P_{C}(u^{k}-\psi_{k} A u^{k})\,,\]
\STATE \textbf{Step 3.} Calculate
\[z^{k}=y^{k}-\psi_{k}(A y^{k}-A u^{k})\,,\]
\STATE \textbf{Step 4.} Calculate
\[x^{k+1}=(1-\sigma \theta_{k} S) q^{k}\,,\]
where $q^{k}=(1-\varphi_{k}) z^{k}+\varphi_{k} T z^{k}$, and update
\begin{equation*}
\psi_{k+1}=\left\{\begin{array}{ll}
\min \left\{\dfrac{\phi\|u^{k}-y^{k}\|}{\|A u^{k}-A y^{k}\|}, \psi_{k}\right\}, & \text { if } A u^{k}-A y^{k} \neq 0; \\
\psi_{k}, & \text { otherwise}.
\end{array}\right.
\end{equation*}
\end{algorithmic}
\end{algorithm}

The following lemma is very useful for studying the convergence of the Algorithm~\ref{alg2}.
\begin{lemma}[\cite{tanarxiv}]\label{lem41}
Suppose that Conditions \ref{con1} and \ref{con3} hold. Let sequence $\{z^{k}\}$ be formulated by Algorithm~\ref{alg2}. Then,  it follows that
\begin{equation*}
\|z^{k}-x^{\dag}\|^{2} \leq\|u^{k}-x^{\dag}\|^{2}-\big(1-\phi^{2} \frac{\psi_{k}^{2}}{\psi_{k+1}^{2}}\big)\|u^{k}-y^{k}\|^{2},\quad \forall x^{\dag} \in \Omega\,,
\end{equation*}
and
\begin{equation*}
\|z^{k}-y^{k}\| \leq \phi \frac{\psi_{k}}{\psi_{k+1}}\|u^{k}-y^{k}\| \,.
\end{equation*}
\end{lemma}

\begin{theorem}\label{thm41}
Suppose that Conditions \ref{con1}--\ref{con5} hold. Then the iterative sequence $\{x^{k}\}$ formulated by Algorithm~\ref{alg2} converges to an element $ x^{\dag}\in \Omega \cap \Gamma $ in norm, where $ x^{\dag} = P_{\Omega \cap \Gamma}(I-\sigma S) x^{\dag} $.
\end{theorem}
\begin{proof}
\noindent{\bf Claim 1.} The sequence $\{x^{k}\}$ is bounded. By Lemma \ref{lem31}, there exists a constant $k_{0} \in \mathbb{N}$ such that $1-\phi^{2} \frac{\psi_{k}^{2}}{\psi_{k+1}^{2}}>0, \forall k \geq k_{0} $. Thanks to Lemma~\ref{lem41}, one sees that
\begin{equation}\label{ee}
\|z^{k}-x^{\dag}\| \leq\|u^{k}-x^{\dag}\|,\quad \forall k \geq k_{0}\,.
\end{equation}
Using the same arguments as in the Theorem~\ref{thm31} of Claim 1, we get that $\{x^{k}\}$ is bounded. Thus, sequences $\{y^{k}\},\{z^{k}\},\{q^{k}\}$ and $\left\{(I-\sigma S) x^{k}\right\}$ are also bounded.

\noindent{\bf Claim 2.}
\[
\begin{aligned}
&\quad\varphi_{k}[1-\vartheta-\varphi_{k}]\|z^{k}-T z^{k}\|^{2}+\big(1-\phi^{2} \frac{\psi_{n^{2}}}{\psi_{k+1}^{2}}\big)\|y^{k}-u^{k}\|^{2}\\ &\leq\|x^{k}-x^{\dag}\|^{2}-\|x^{k+1}-x^{\dag}\|^{2}+\theta_{k} Q_{4}\,, \quad \forall k \geq k_{0}
\end{aligned}
\]
for some $ Q_{4}>0 $. From \eqref{eqaa}, \eqref{eqo} and  Lemma~\ref{lem41}, we can show that
\[
\begin{aligned}
\|x^{k+1}-x^{\dag}\|^2
\leq& \|z^{k}-x^{\dag}\|^{2}+\varphi_{k}[\varphi_{k}-(1-\vartheta)]\|T z^{k}-z^{k}\|^{2}+ \theta_{k}Q_{2}\\
\leq&\|x^{k}-x^{\dag}\|^{2}-\big(1-\phi^{2} \frac{\psi_{k}^{2}}{\psi_{k+1}^{2}}\big)\|y^{k}-u^{k}\|^{2}+ \theta_{k}Q_{4} \\
&+\varphi_{k}[\varphi_{k}-(1-\vartheta)]\|T z^{k}-z^{k}\|^{2}\,, \quad \forall k \geq k_{0}\,,
\end{aligned}
\]
where $Q_{4}:=Q_{2}+Q_{3}$, both $ Q_{2} $ and $ Q_{3} $ are defined in Claim 2 of Theorem~\ref{thm31}.

\noindent{\bf Claim 3.}
\begin{equation*}
\|x^{k+1}-x^{\dag}\|^2
\leq(1-\gamma \theta_{k})\|x^{k}-x^{\dag}\|^2+\gamma\theta_{k} \big[\frac{2\sigma}{\gamma}\langle S x^{\dag}, x^{\dag}-x^{k+1}\rangle+\frac{3Q\xi_{k}}{\gamma\theta_{k} }\|x^{k}-x^{k-1}\|\big],\, \forall k \geq k_{0}\,,
\end{equation*}
This result can be obtained using the same arguments as in Theorem~\ref{thm31} of Claim~3.

\noindent{\bf Claim 4.} The sequence $\{\|x^{k}-x^{\dag}\|^{2}\}$ converges to zero. The proof is similar to Claim 4 in Theorem~\ref{thm31}. We leave it for the reader to check.
\end{proof}

Now, we give a special case of Algorithm~\ref{alg2}. When $ S(x)=x-f(x) $ in Theorem~\ref{thm41}, where mapping $f: H \rightarrow H$ is $ \rho $-contraction. It can be easily verified that mapping $S: H \rightarrow H$ is $(1+\rho)$-Lipschitz continuous and $(1-\rho)$-strongly monotone. In this situation, by picking $\sigma=1$, we get a new self-adaptive inertial viscosity-type Tseng's extragradient algorithm for solving~\eqref{VIPFPP}. Similar to corollary~\ref{coro1}, we can get the following results immediately.
\begin{corollary}\label{coro2}
Suppose that mapping $A: H \rightarrow H$ is $L$-Lipschitz continuous monotone, mapping $T: H \rightarrow H$ is $\vartheta$-demicontractive such that $(I-T)$ is demiclosed at zero and mapping $f:~{H} \rightarrow {H}$ is $ \rho $-contractive with $\rho \in[0, \sqrt{5}-2)$. Give  $ \xi>0 $, $\psi_{1}>0$, $\phi \in(0,1)$. Let sequence $ \{\zeta_{k}\} $ be positive numbers such that $\lim_{k \rightarrow \infty} \frac{\zeta_{k}}{\theta_{k}}=0$, where $ \{\theta_{k}\}\subset (0,1) $ satisfies   $\sum_{k=0}^{\infty} \theta_{k}=\infty $ and $\lim _{k \rightarrow \infty} \theta_{k}=0$. Let $\left\{\varphi_{k}\right\}$ be a real sequence such that $\varphi_{k} \subset(a,1-\vartheta)$ for some $a>0$. Let $x^{0},x^{1} \in H$ and $\{x^{k}\}$ be defined by
\begin{equation}\label{r}
\left\{\begin{aligned}
&u^{k}=x^{k}+\xi_{k}(x^{k}-x^{k-1})\,,\\
&y^{k}=P_{C}(u^{k}-\psi_{k} A u^{k})\,, \\
&z^{k}=y^{k}-\psi_{k}(A y^{k}-A u^{k})\,, \\
&q^{k}=(1-\varphi_{k}) z^{k}+\varphi_{k} T z^{k}\,, \\
&x^{k+1}=(1-\theta_{k}) q^{k}+\theta_{k} f(q^{k})\,, 
\end{aligned}\right.
\end{equation}
where inertial parameter $ \xi_{k} $ and step size $ \psi_{k} $ are defined \eqref{alpha} and \eqref{lambda2}, respectively. Then the iteration sequence $\{x^{k}\}$ formulated by \eqref{r} converges to $x^{\dag} \in \Omega \cap \Gamma$ in norm, where $x^{\dag}=P_{\Omega \cap \Gamma} \circ f(p)$.
\end{corollary}

\begin{remark}
\begin{enumerate}
\item Set $ S(x)=x-f(x) $ in Theorem~\ref{thm31} and select $ S(x)=x-x^{0} $ in Theorem~\ref{thm41}. We can get two new algorithms to seek the common solution of problem \eqref{VIP} and problem \eqref{FPP}. Note that these algorithms all obtain strong convergence results in Hilbert spaces. Furthermore, they can work without the prior information about the Lipschitz constant of the operator.
\item The algorithms proposed in this paper improve and extend some recent results in the literature~\cite{KS,TVNA,STEGM}. Both of our algorithms embed inertial terms and use new iteration steps, which makes them faster and more flexible. In addition, it is worth noting that $ T $ in Algorithms~\eqref{HSEGM} and \eqref{STEGM} is a quasi-nonexpansive mapping, but ours is a demicontractive mapping. Therefore, our algorithms have a wider range of applications.
\end{enumerate}
\end{remark}
\section{Numerical examples}\label{sec4}
In this section, we provide some computational tests to illustrate the  numerical behavior of our proposed algorithms (Algorithm~\ref{alg1} (iSTEGM), Algorithm~\ref{alg2} (iSSEGM)) and compare them with some existing strong convergence methods, including the Halpern subgradient extragradient method \eqref{HSEGM} \cite{KS}, the viscosity-type subgradient extragradient method \eqref{VSEGM} \cite{TVNA}, the viscosity-type Tseng's extragradient method \eqref{VTEGM} \cite{TVNA}, and the self-adaptive Tseng's extragradient method \eqref{STEGM} \cite{STEGM}. We use the FOM Solver~\cite{FOM} to effectively calculate the projections onto $ C $ and $ H_{k} $. All the programs were implemented in MATLAB 2018a on a Intel(R) Core(TM) i5-8250T CPT @ 1.60GHz computer with RAM 8.00 GB.

Our parameters are set as follows. In all algorithms, set $ \theta_{k}=1/(k+1) $ and $ \varphi_{k}=k/(2k+1) $. For the proposed algorithms and the algorithms \eqref{VSEGM} and \eqref{VTEGM}, we choose $ \psi_{1}=0.9 $, $ \phi=0.5 $. Setting $ f(x)=0.5x $ in the algorithms \eqref{VSEGM} and \eqref{VTEGM}. Take $ \sigma=0.5 $, $ \xi=0.4 $, $ \zeta_{k}=1/(k+1)^2 $ in our proposed algorithms. For the algorithm \eqref{STEGM}, we select $ \alpha=0.5 $, $ \ell=0.5 $, $ \phi=0.4 $ and $ \sigma=0.5 $.  For the algorithm \eqref{HSEGM}, we pick out the step size as $ \psi= 0.99/L$. In our numerical examples, when the number of iterations is the same, we use the runtime in seconds to measure the computational performance of all algorithms. In addition, the solution $ x^{*} $ of the problems are known. Thus, we use the function $ D_{k}=\|x^{k}-x^{*}\| $ to measure the $ k $-th iteration error. Obviously, $ D_{k}=0 $ means that $ x^{k} $ converges to $ x^{*} $, which can be regarded as an approximate solution to the problems.

\begin{example}\label{ex0}
In first example, we consider a simple two-dimensional numerical test. Let the nonlinear mapping $A: \mathbb{R}^{2} \rightarrow \mathbb{R}^{2}$ be defined as follows:
\[
A(x, y)=(x+y+\sin x ;-x+y+\sin y)\,.
\]
It is easy to verify that mapping $A$ is Lipschitz continuous monotone with modulus $ L=3 $. Assume that the feasible set $C$ is a two-dimensional box with lower bounds $ l_{i}=[-1;-1] $ and upper bounds $ u_{i}=[1;1] $. Then the projection of a point $ x_{i} \in \mathbb{R}^{2} $ on this box can be calculated explicitly by the following formula: $ P_{C} (x)_{i} = \min\{ u_{i}, \max \{ l_{i}, x_{i}\}\}$. Moreover, the mapping $ T: \mathbb{R}^{2} \rightarrow \mathbb{R}^{2} $ is defined by $T x=\|D\|^{-1} D x$, where $ D $ is a second-order matrix, defined as $ D=[1, 0; 0, 2] $. The mapping $ S: \mathbb{R}^{2} \rightarrow \mathbb{R}^{2} $ is selected as $ Sx = 0.5 x $. It can be easily seen that that mapping $ T $ is $0$-demicontractive and the mapping $ S $ is Lipschitz continuous and strongly monotone. We can easily find the solution of this problem as $x^{*}=(0,0)^{\mathsf{T}} $.  In order to verify the effectiveness of the algorithm, we select four different initial values $ x^{0} = x^{1} $ in MATLAB, namely, (Case~I): $ x^{1} = \emph{rand(2,1)} $, (Case~II): $ x^{1} = \emph{5rand(2,1)} $, (Case~III): $ x^{1} = \emph{10rand(2,1)} $, (Case~IV): $ x^{1} = \emph{20rand(2,1)} $, and the maximum iteration $ 400 $ as the common stop criterion. The numerical results are plotted in Figs.~\ref{figEX1_1}--\ref{figEX1_20}.
\begin{figure}[htbp]
	\centering
	\includegraphics[scale=0.5]{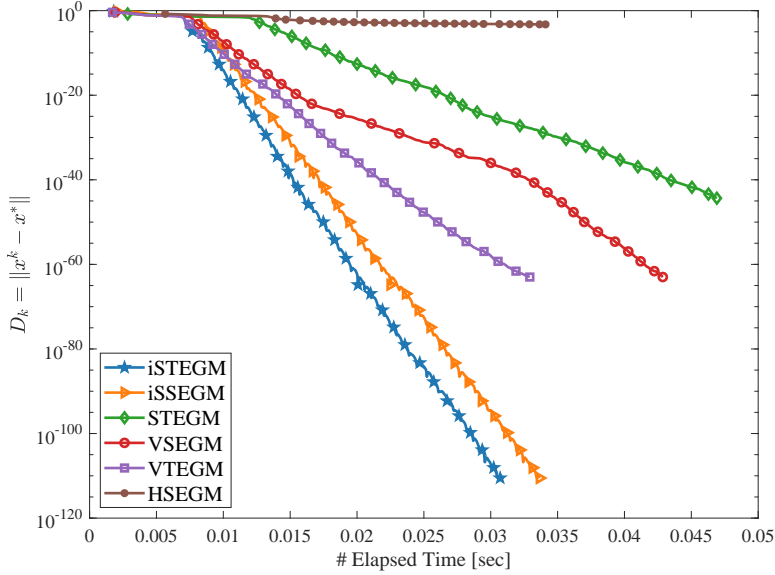}
	\caption{Numerical results of Example~\ref{ex0} for $ x^{1} = \emph{rand(2,1)} $}
	\label{figEX1_1}
\end{figure}
\begin{figure}[htbp]
	\centering
	\includegraphics[scale=0.5]{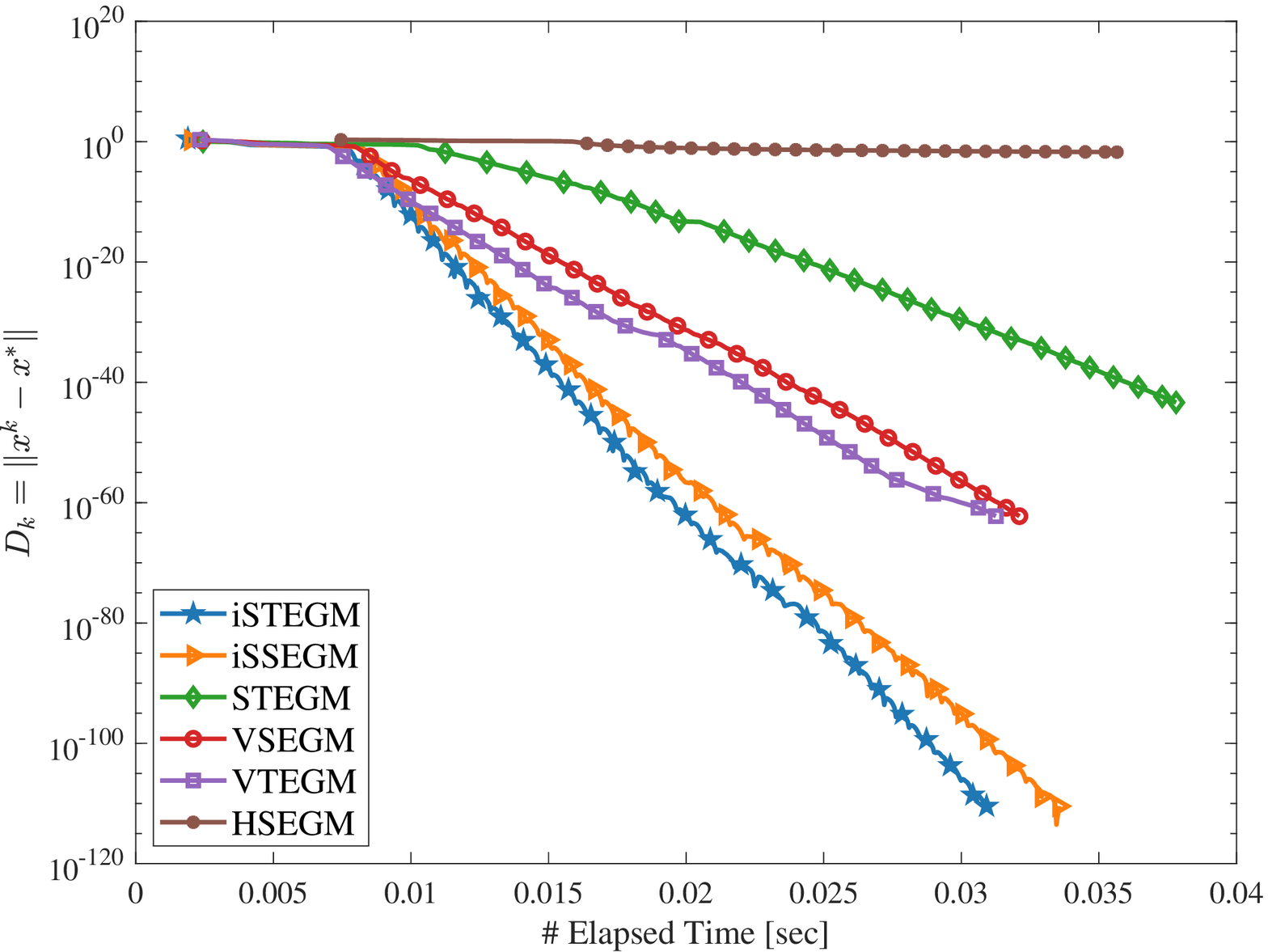}
	\caption{Numerical results of Example~\ref{ex0} for $ x^{1} = \emph{5rand(2,1)} $}
	\label{figEX1_5}
\end{figure}
\begin{figure}[htbp]
	\centering
	\includegraphics[scale=0.5]{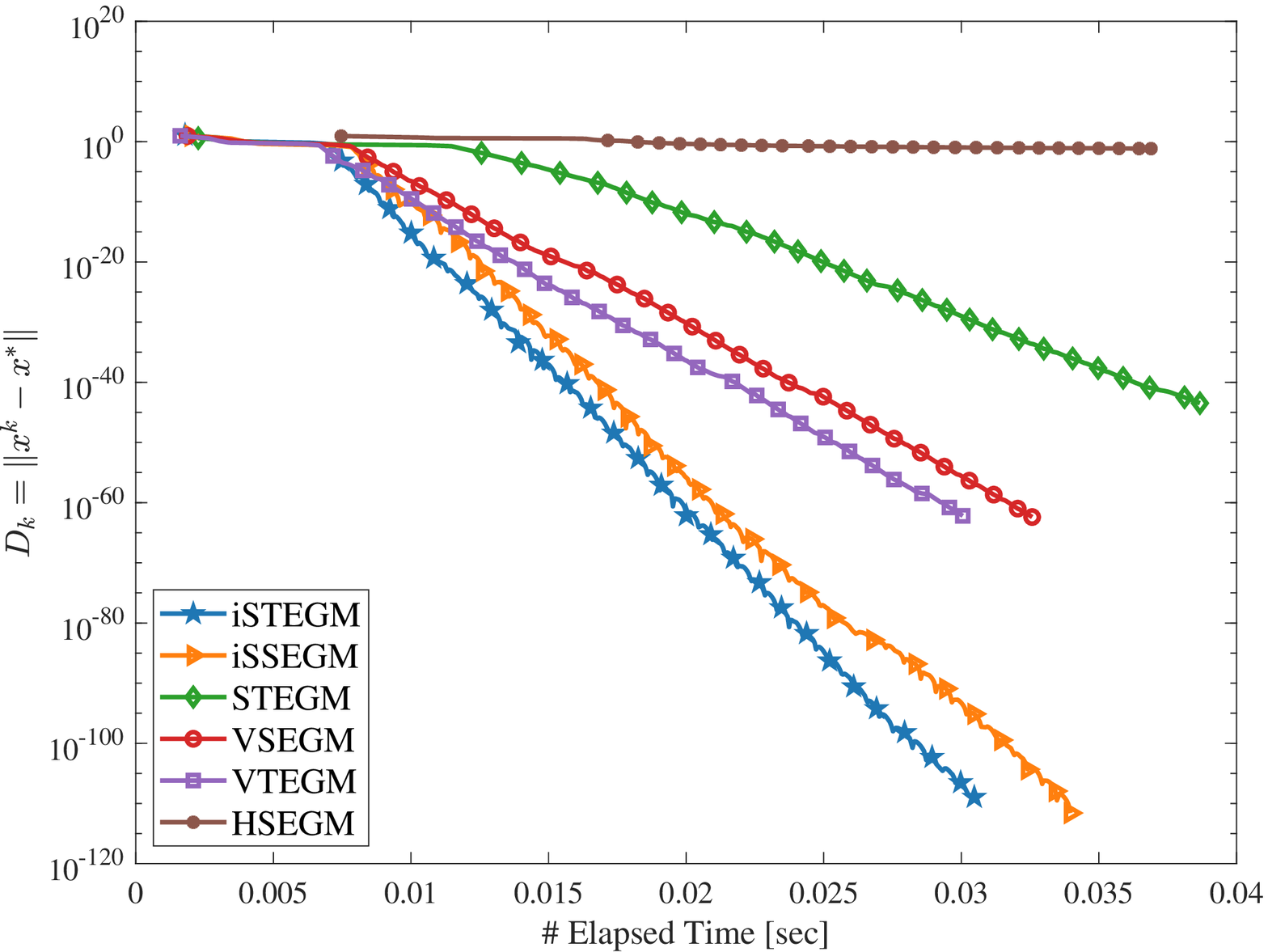}
	\caption{Numerical results of Example~\ref{ex0} for $ x^{1} = \emph{10rand(2,1)} $}
	\label{figEX1_10}
\end{figure}
\begin{figure}[htbp]
	\centering
	\includegraphics[scale=0.5]{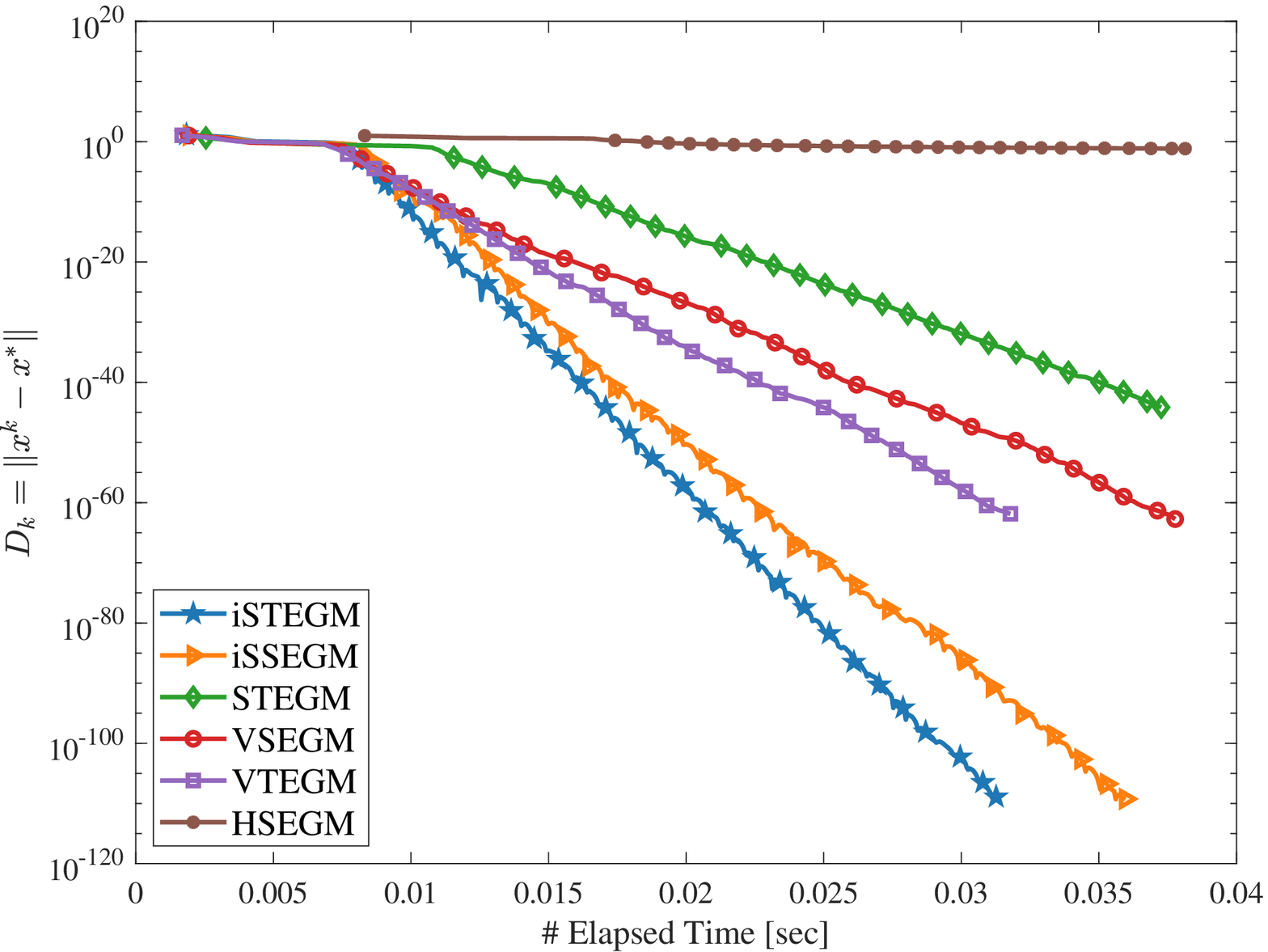}
	\caption{Numerical results of Example~\ref{ex0} for $ x^{1} = \emph{20rand(2,1)} $}
	\label{figEX1_20}
\end{figure}
\end{example}

\begin{example}\label{ex1}
In the second example, we consider the form of linear operator $ A: R^{n}\rightarrow R^{n} $ ($ n=100, 200 $) as follows: $A(x)=Gx+g$, where $g\in R^n$ and $G=BB^{\mathsf{T}}+M+E$, matrix $B\in R^{n\times n}$, matrix $M\in R^{n\times n}$ is skew-symmetric, and matrix $E\in R^{n\times n}$ is diagonal matrix whose diagonal terms are non-negative (hence $ G $ is positive symmetric definite). We choose the feasible set as $C=\left\{x \in {R}^{n}:-2 \leq x_{i} \leq 5, \, i=1, \ldots, n\right\}$.  It can be easily checked that mapping $A$ is Lipschitz continuous monotone and its Lipschitz constant $ L=  \|G\| $.  In this numerical example, both $B, E$ entries are randomly created in $[0,2]$, $M$ is generated randomly in $[-2,2]$ and $ g = 0 $. Let $T: H \rightarrow H$ and $S: H \rightarrow H$ be provided by $T x=0.5 x$ and $S x=0.5 x$, respectively.  We obtain the solution to the problem is $ x^{*}=\{\mathbf{0}\} $. The maximum iteration $ 400 $ as a common stopping criterion and the initial values $ x^{0} = x^{1} $ are randomly generated by \emph{rand(n,1)} in MATLAB. The numerical results with elapsed time are described in Figs.~\ref{fig50_time}--\ref{fig200_time}.
\begin{figure}[htbp]
\centering
\includegraphics[scale=0.5]{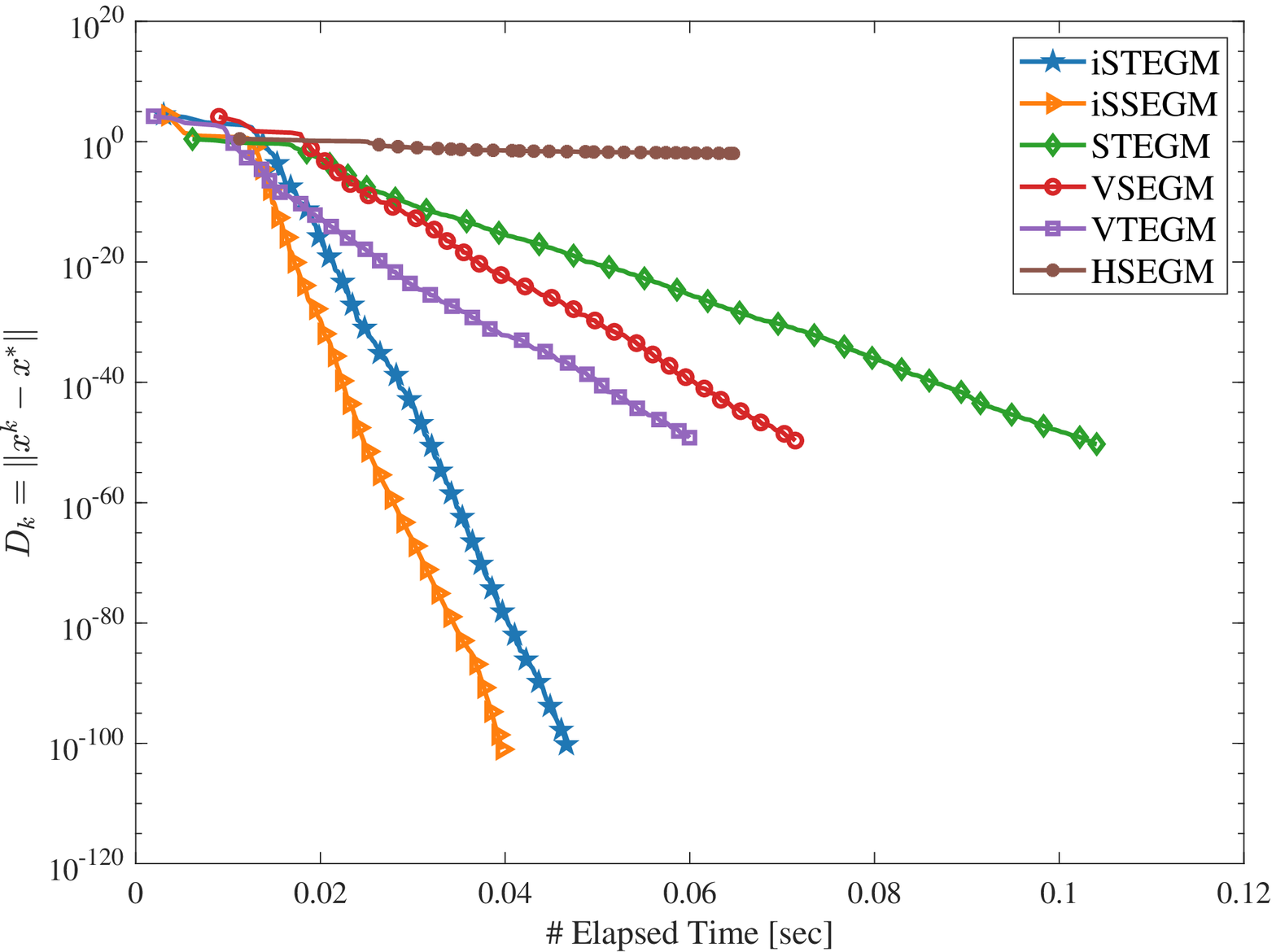}
\caption{Numerical results of Example~\ref{ex1} when $ n=50 $}
\label{fig50_time}
\end{figure}
\begin{figure}[htbp]
\centering
\includegraphics[scale=0.5]{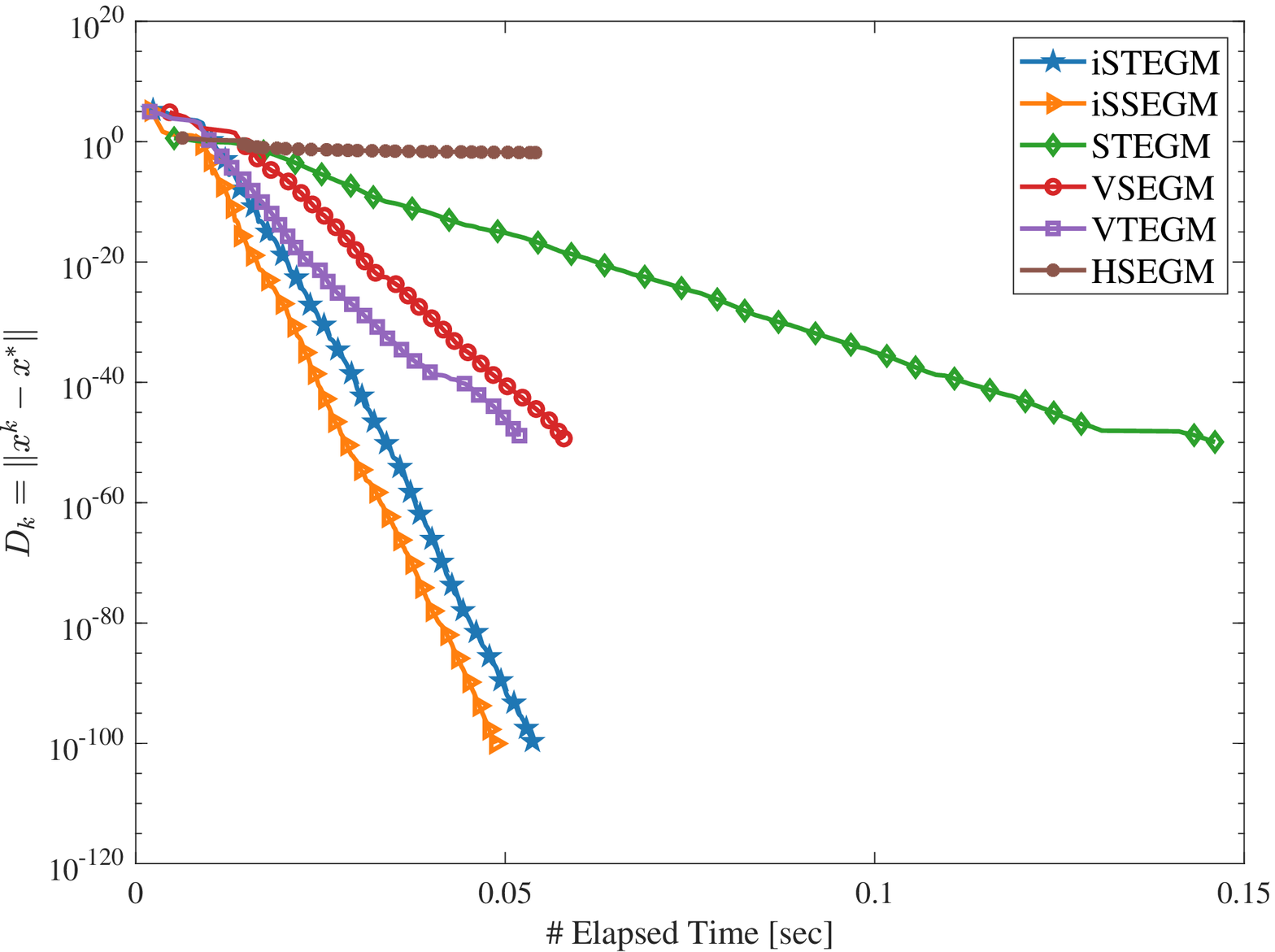}
\caption{Numerical results of Example~\ref{ex1} when $ n=100 $}
\label{fig100_time}
\end{figure}
\begin{figure}[htbp]
\centering
\includegraphics[scale=0.5]{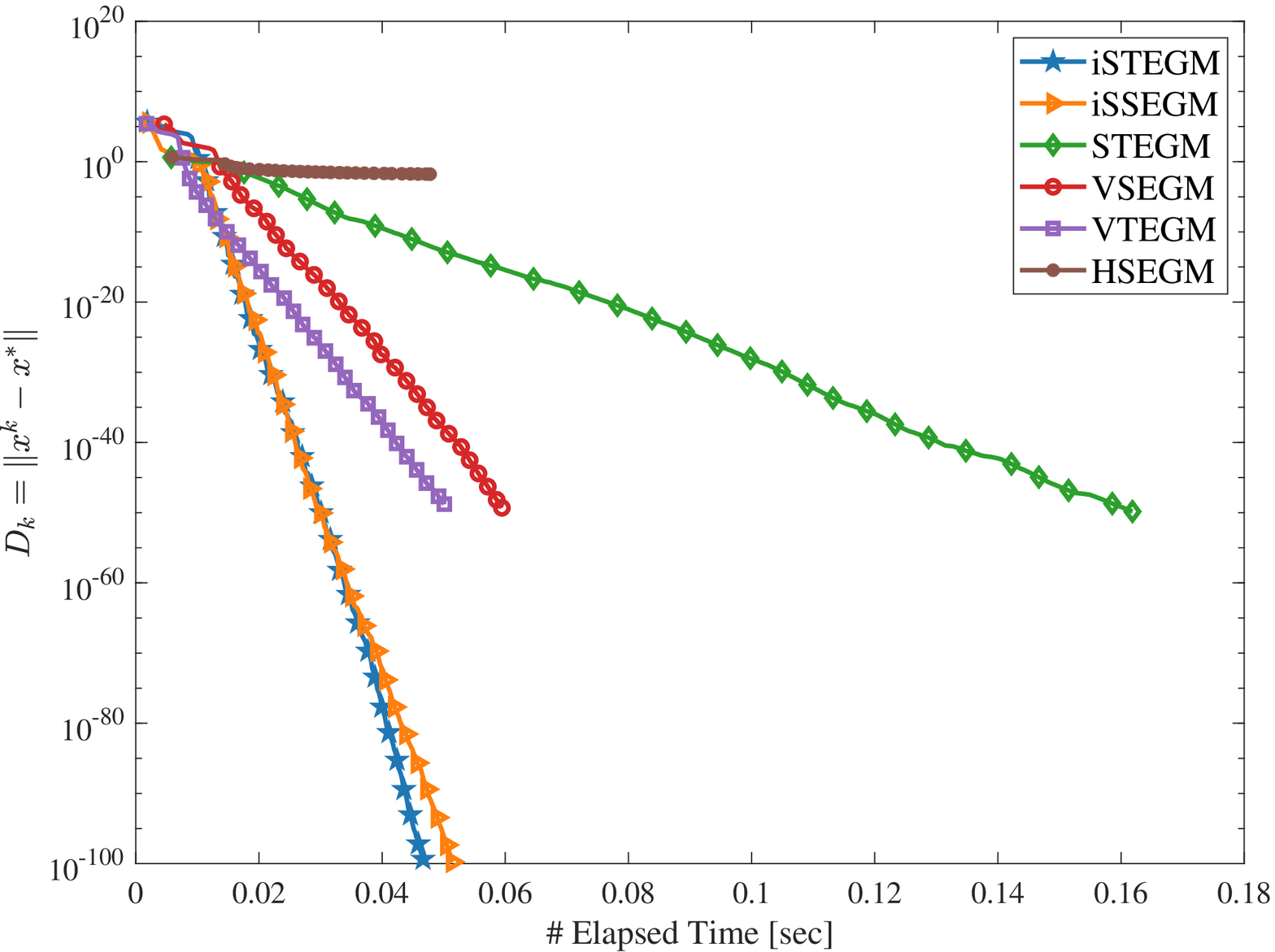}
\caption{Numerical results of Example~\ref{ex1} when $ n=150 $}
\label{fig150_time}
\end{figure}
\begin{figure}[htbp]
\centering
\includegraphics[scale=0.5]{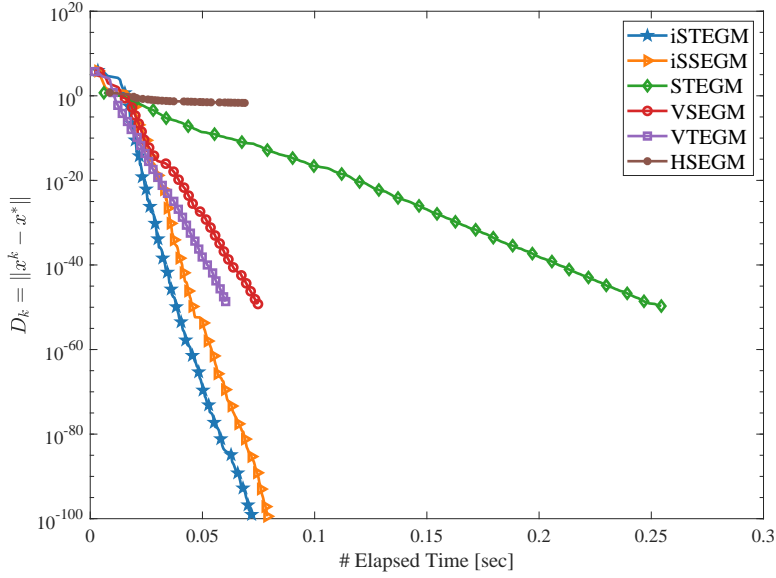}
\caption{Numerical results of Example~\ref{ex1} when $ n=200 $}
\label{fig200_time}
\end{figure}
\end{example}

\begin{example}\label{ex2}
In this numerical example, we focus on a case in Hilbert space $H=L^{2}([0,1])$. Its inner product and induced norm are defined as $\langle x, y\rangle:=\int_{0}^{1} x(t) y(t) \mathrm{d} t$ and $\|x\|:=(\int_{0}^{1}|x(t)|^{2} \mathrm{d} t)^{1 / 2}$, respectively. The unit ball $C:=\{x \in H:\|x\| \leq 1\}$ as the feasible set. Let the operator $A: C \rightarrow H$ be generated as follows:
\[
(A x)(t)=\max \{0, x(t)\}=\frac{x(t)+|x(t)|}{2}\,.
\]
It can be easily verified that $ A $ is monotone and Lipschitz continuous with modulus $ L=1 $. Moreover, the projection onto the feasible set $ C $ is explicit, and we can use the following formula to calculate the projection:
\[
P_{C}(x)=\left\{\begin{array}{ll}
\frac{x}{\|x\|_{L^{2}}}, & \text { if }\|x\|_{L^{2}}>1; \\
x, & \text { if }\|x\|_{L^{2}} \leq 1.
\end{array}\right.
\]
We choose the mapping $T: L^{2}([0,1]) \rightarrow L^{2}([0,1])$ is of form $ (T x)(t)=\int_{0}^{1} t x(s) d s, t \in[0,1] $. A simple computation indicates that $T$ is $ 0 $-demicontractive and demiclosed at zero. Let mapping $S: H \rightarrow H$ be taken as $(S x)(t)=0.5 x(t),  t \in[0,1]$.  It can be easily proved that the mapping $ S $ is strongly monotone and Lipschitz continuous. The solution to this problem is $x^{*}(t)=0 $. Our parameter settings are the same as in Example~\ref{ex1}, and the maximum iteration $ 50 $ is used as the stopping criterion. With four types of starting points: (Case~I) $x^{0}(t)=x^{1}(t)=t^2$, (Case~II) $ x^{0}(t)=x^{1}(t)=2^t $, (Case~III) $ x^{0}(t)=x^{1}(t)=e^t $ and (Case~IV) $x^{0}(t)=x^{1}(t)=t+0.5\cos(t)$.  The numerical behaviors of $D_{k}=\|x^{k}(t)-x^{*}(t)\|$ formulated by all the algorithms are shown in Figs.~\ref{figt2_time}--\ref{figcosq_time}.
\begin{figure}[htbp]
\centering
\includegraphics[scale=0.5]{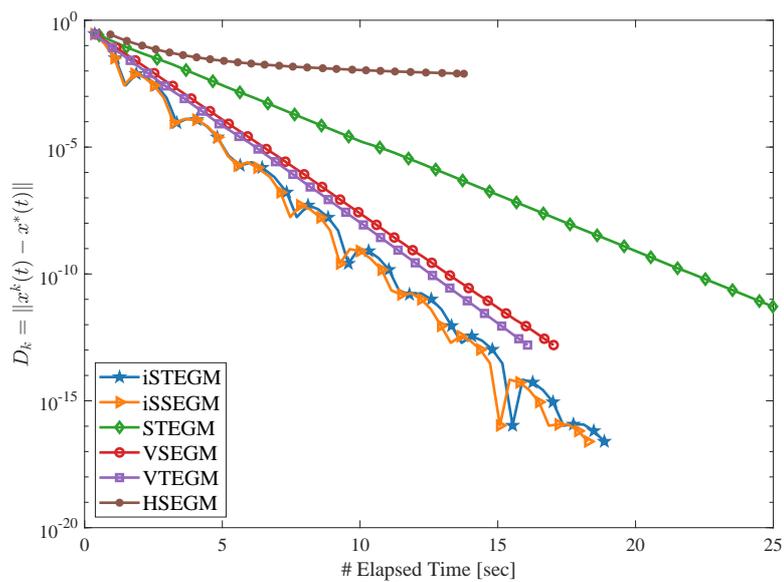}
\caption{Numerical results of Example~\ref{ex2} when $ x^{0}(t)=x^{1}(t)=t^2 $}
\label{figt2_time}
\end{figure}
\begin{figure}[htbp]
\centering
\includegraphics[scale=0.5]{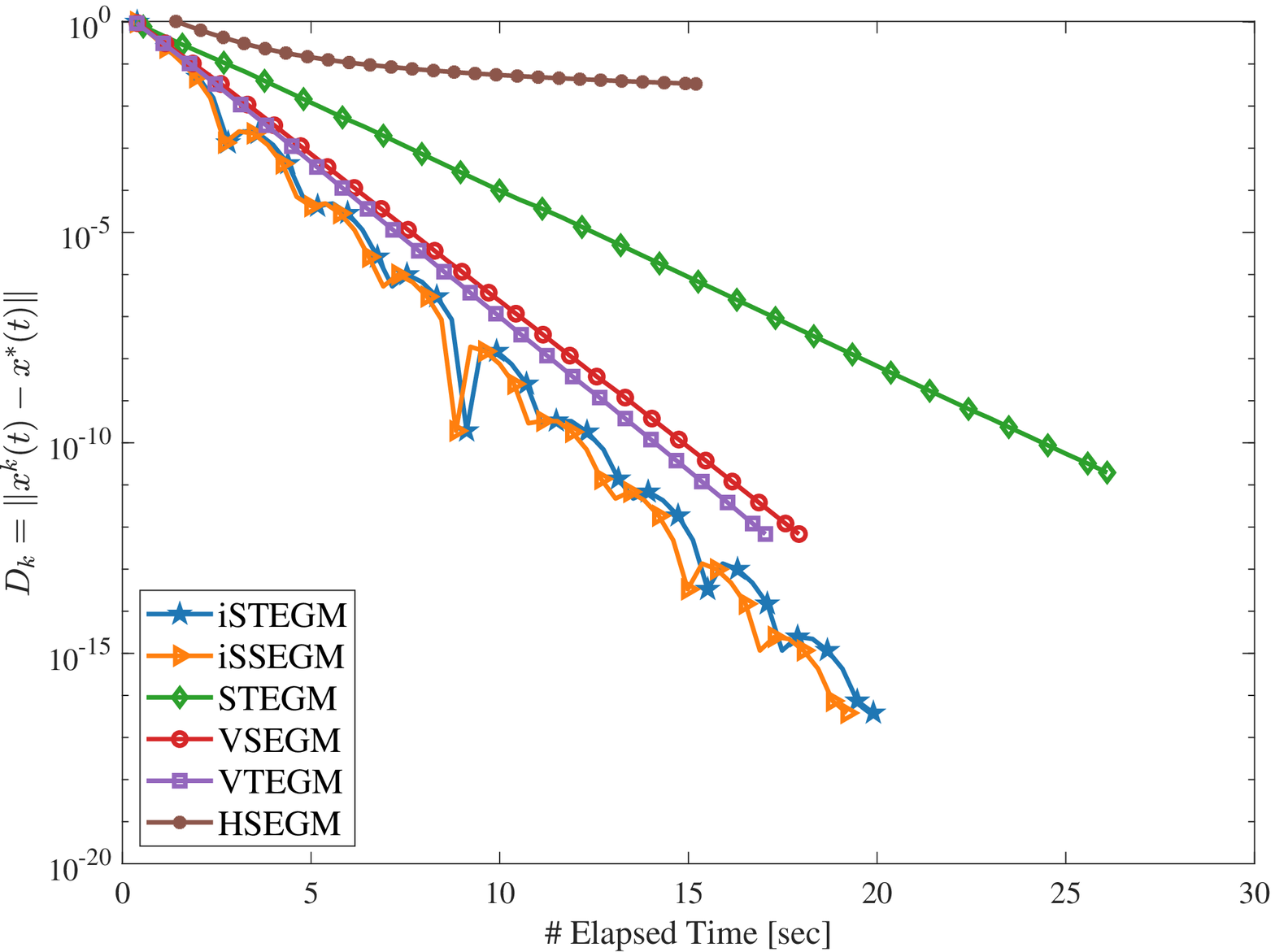}
\caption{Numerical results of Example~\ref{ex2} when $ x^{0}(t)=x^{1}(t)=2^t $}
\label{fig2t_time}
\end{figure}
\begin{figure}[htbp]
\centering
\includegraphics[scale=0.5]{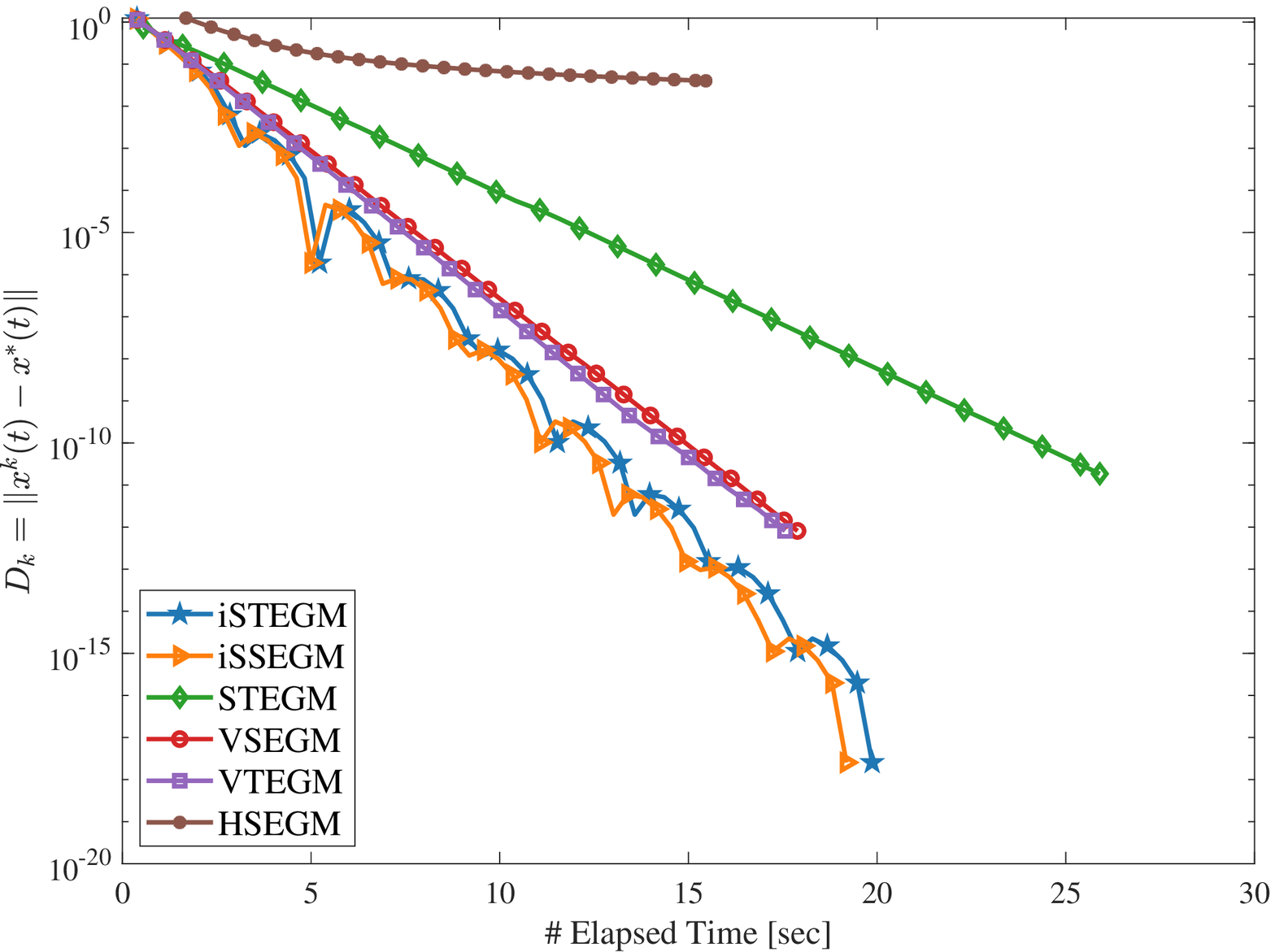}
\caption{Numerical results of Example~\ref{ex2} when $ x^{0}(t)=x^{1}(t)=e^t $}
\label{figet_time}
\end{figure}
\begin{figure}[htbp]
\centering
\includegraphics[scale=0.5]{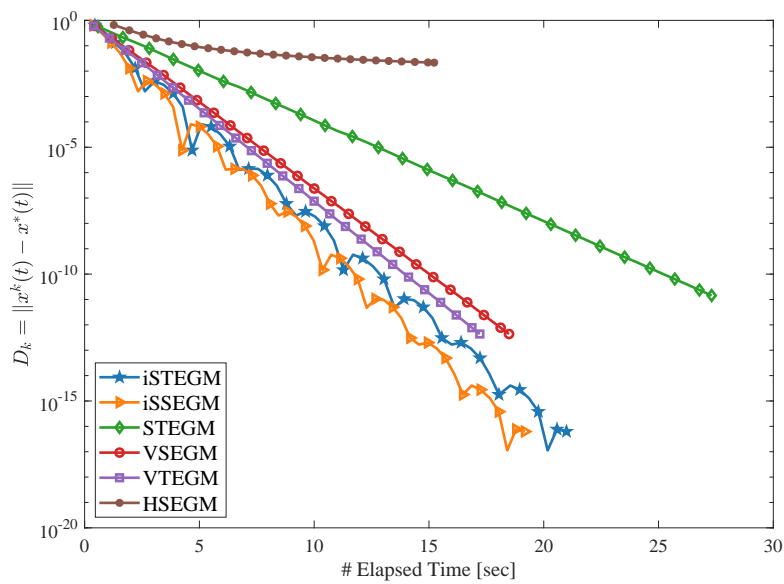}
\caption{Numerical results of Example~\ref{ex2} when $ x^{0}(t)=x^{1}(t)=t+0.5\cos(t) $}
\label{figcosq_time}
\end{figure}
\end{example}

\begin{remark}
\begin{enumerate}
\item From Figs.~\ref{figEX1_1}--\ref{figcosq_time}, we know that our proposed algorithms outperformance some existing algorithms in the literature. These results are independent of the selection of initial values and the size of dimensions. Note that our algorithms converge very quickly, and there are still some oscillations since the inertial effect. 
\item The maximum number of iterations we choose is only $ 400 $. It should be noted that the iteration error of Algorithm~\eqref{HSEGM} is very big. In actual applications, it may require more iterations to meet the accuracy requirements. Furthermore, we point out that since the Algorithm~\eqref{STEGM} uses the Armijo-like step size rule, which leads to taking more execution time.
\item In our future work, we will improve the generality of the operators involved, for example, consider the operator $ A $ is pseudo-monotone and uniformly continuous. We will also consider how to reduce the oscillation effect caused by the inertial term. In addition, the wide application in image processing and machine learning deserves further consideration.
\end{enumerate}
\end{remark}

\section{Conclusions}\label{sec5}
In this study, we investigated the problem of seeking a common solution to the variational inequality problem involving monotone and Lipschitz continuous mapping and the fixed point problem with a demicontractive mapping. We proposed two inertial extragradient methods with new step sizes to compute the approximate solutions of problems in a Hilbert space. The strong convergence of the suggested methods is established under standard and suitable conditions.  Finally, some computational tests are given to explain our convergent results.  Our algorithms obtained in this paper improve and summarizes some of the recent results in the literature.

\end{document}